\def\submission{0}

\documentclass{article}

\usepackage[english]{babel}
\usepackage[utf8]{inputenc}
\usepackage[T1]{fontenc}
\usepackage[a4paper,left=3cm,right=3cm,top=3cm,bottom=3.2cm]{geometry} 
\usepackage{tikz}
\usetikzlibrary[topaths] 

\usepackage{amsmath, mathrsfs, amssymb, amsthm, mathtools, bbm, bm} 

\usepackage{caption,float}

\usepackage{makecell} 
\usepackage{enumitem} 

\newtheorem{theorem}{Theorem}
 
\newtheorem{definition}{Definition}
\newtheorem{lemma}{Lemma}
\newtheorem{remark}{Remark}
\newtheorem{corollary}{Corollary}
\newtheorem{proposition}{Proposition}
\newtheorem{assumption}{Assumption}
\newtheorem{claim}{Claim}

\newtheorem*{theorem*}{Theorem}
\newtheorem*{example*}{Example} 
\newtheorem*{definition*}{Definition}
\newtheorem*{lemma*}{Lemma}
\newtheorem*{remark*}{Remark}
\newtheorem*{corollary*}{Corollary}
\newtheorem*{proposition*}{Proposition}
\newtheorem*{assumption*}{Assumption}
\newtheorem*{claim*}{Claim}

\newtheoremstyle{TheoremNum}
        {\topsep}{\topsep}              
        {\itshape}                      
        {}                              
        {\bfseries}                     
        {.}                             
        { }                             
        {\thmname{#1}\thmnote{ \bfseries #3}}
\theoremstyle{TheoremNum}

\newtheoremstyle{LemmaNum}
        {\topsep}{\topsep}              
        {\itshape}                      
        {}                              
        {\bfseries}                     
        {.}                             
        { }                             
        {\thmname{#1}\thmnote{ \bfseries #3}}
\theoremstyle{LemmaNum}

\usepackage{hyperref}

\usepackage{float}

\usepackage{enumitem}

\usepackage{makecell} 

\usepackage{subfig}
\usepackage{float}

\usepackage{enumitem} 

\usepackage{algorithm} 
\usepackage{algcompatible} 

\newcommand{\F}{\mathcal{F}}

\renewcommand{\S}{ \mathbb{ S } }

\newcommand{\x}{ \mathbf{ x } } 
\newcommand{\y}{ \mathbf{ y } }

\newcommand{\St}{\mathrm{St}}

\renewcommand{\P}{\mathbf{P}}

\newcommand{\V}{\mathbf{V}}

\newcommand{\0}{\bm{0}}

\renewcommand{\u}{\mathbf{u}}

\newcommand{\1}{\text{\raisebox{.5pt}{\textcircled{\raisebox{-.9pt} {1}}}}
} 
\newcommand{\2}{\text{\raisebox{.5pt}{\textcircled{\raisebox{-.9pt} {2}}}}
}

\newcommand{\Ind}{ \mathbb{I} }

\renewcommand{\Pr}{ \mathbb{P} }

\newcommand{\R}{\mathbb{R}}

\newcommand{\E}{\mathbb{E}}

\renewcommand{\v}{\mathbf{v}}

\renewcommand{\[}{\left[ }
\renewcommand{\]}{\right] }

\newcommand{\<}{\left< }
\renewcommand{\>}{\right> }

\renewcommand{\(}{\left( }
\renewcommand{\)}{\right) }

\newcommand{\wh}{\widehat }

\newcommand{\Loj}{\L ojasiewicz } 

\newcommand{\dom}{\mathrm{dom} } 

\def\submission{0} 

\begin{document} 

\title{Convergence Rates of Stochastic Zeroth-order Gradient Descent for \L ojasiewicz Functions} 

\author{Tianyu Wang\footnote{wangtianyu@fudan.edu.cn} \quad and \quad Yasong Feng} 

\date{} 

\maketitle 

\begin{abstract} 
    We prove convergence rates of Stochastic Zeroth-order Gradient Descent (SZGD) algorithms for \L ojasiewicz functions. 
The SZGD algorithm iterates as 
\begin{align*} 
    \mathbf{x}_{t+1} = \mathbf{x}_t - \eta_t \widehat{\nabla} f (\mathbf{x}_t), \qquad t = 0,1,2,3,\cdots , 
\end{align*} 
where $f$ is the objective function that satisfies the \L ojasiewicz inequality with \L ojasiewicz exponent $\theta$, $\eta_t$ is the step size (learning rate), and $ \widehat{\nabla} f (\mathbf{x}_t) $ is the approximate gradient estimated using zeroth-order information only. 
We show that, for smooth \L ojasiewicz functions, the sequence $\{ \mathbf{x}_t \}_{t\in\mathbb{N}}$ generated by SZGD converges to a point $\mathbf{x}_\infty$ almost surely, and $\mathbf{x}_\infty$ is a critical point of $f$. 
If $\theta \in (0,\frac{1}{2}]$, $ \{ f (\mathbf{x}_t) - f (\mathbf{x}_\infty) \}_{t \in \mathbb{N} } $, $ \{ \sum_{s=t}^\infty \| \mathbf{x}_{s+1} - \mathbf{x}_{s} \|^2 \}_{t \in \mathbb{N} } $ and $ \{ \| \mathbf{x}_t - \mathbf{x}_\infty \| \}_{t \in \mathbb{N} } $ ($\| \cdot \|$ is the Euclidean norm) converge to zero linearly in expectation. 
If $\theta \in (\frac{1}{2}, 1)$, then $ \{ f (\mathbf{x}_t) - f (\mathbf{x}_\infty) \}_{t \in \mathbb{N} } $ (and $ \{ \sum_{s=t}^\infty \| \mathbf{x}_{s+1} - \mathbf{x}_s \|^2 \}_{t \in \mathbb{N} } $) converges to zero at rate $O \left( t^{\frac{1}{1 - 2\theta}}  \right) $ in expectation; $ \{ \| \mathbf{x}_{t} - \mathbf{x}_\infty \| \}_{t \in \mathbb{N} } $ converges to zero at rate $O \left( t^{\frac{1-\theta}{1-2\theta}} \right) $ in expectation. Our results show that $ \{ f (\mathbf{x}_t) - f (\mathbf{x}_\infty) \}_{t \in \mathbb{N} } $ can converge faster than $ \{ \| \mathbf{x}_t - \mathbf{x}_\infty \| \}_{t \in \mathbb{N} }$. 

Also, we show that for convex nonsmooth \L ojasiewicz functions with \L ojasiewicz exponent $\theta \in (\frac{1}{2},1)$, the sequence $ \{ f (\mathbf{x}_t) - f (\mathbf{x}_\infty) \}_{t \in \mathbb{N}} $ generated by the proximal algorithm converges at rate $O \left( t^{\frac{1}{1-2\theta}} \right) $. This rate is faster than the state-of-the-art convergence rate of $ \{ \| \mathbf{x}_t - \mathbf{x}_\infty \| \}_{t \in \mathbb{N} } $ for $\theta \in (\frac{1}{2},1)$. 
 
\end{abstract} 

\section{Introduction}

    


Zeroth order optimization is a central topic in optimization and related fields. Algorithms for zeroth order optimization find important real-world applications, since often times in practice, we cannot directly access the derivatives of the objective function. 
To optimize the function in such scenarios, one can estimate the gradient/Hessian first and deploy first/second order algorithms with the estimated derivatives. Previously, many authors have considered this problem. Yet stochastic zeroth order methods for \Loj functions have not been carefully investigated (See Section \ref{sec:related-works} for more discussion).

\Loj functions are real-valued functions that satisfy the \Loj inequality \cite{loj1963}. 
The \Loj inequality generalizes the Polyak--\Loj inequality \cite{polyak1963}, and is a special case of the Kurdyka--\Loj (KL) inequality \cite{klz1995,Kurdyka1998OnGO}. 
Such functions may give rise to spiral gradient flow even if smoothness and convexity are assumed \cite{https://doi.org/10.1112/blms.12586}. Also, \Loj functions may not be convex. The compatibility with nonconvexity has gained them increasing amount of attention, due to the surge in nonconvex objectives from machine learning and deep learning. 
Indeed, the \Loj inequality can well capture the local landscape of neural network losses, since  some good local approximators for neural network losses, including polymonials and semialgebraic functions, locally satisfy the \Loj inequality. 


Previously, the understanding of \Loj functions have been advanced by many researchers \cite{polyak1963,loj1963,klz1995,Kurdyka1998OnGO,kurdyka2000proof,bolte2007lojasiewicz,attouch2009convergence,li2018calculus}. In their classic work, \cite{attouch2009convergence} proved the state-of-the-art convergence rate for $\{ \| \x_t - \x_\infty \| \}_t$, where $\{ \x_t \}_t$ is the sequence generated by the proximal algorithm, $\x_\infty$ is the limit of $\{ \x_t \}_t$ that is also a critical point of the objective $f$, and $\| \cdot \|$ denotes the Euclidean norm. 

\subsubsection*{Convergence Rates of SZGD} 

In this paper, we study the performance of gradient descent with estimated gradient for (smooth) \Loj functions. 
In particular, we study algorithms governed by the following rule 
\begin{align} 
    \x_{t+1} = \x_t - \eta_t \wh{\nabla} f_{\V_t}^{\delta_t} (\x_t), \qquad t = 0,1,2,\cdots , 
    \label{eq:gd} 
\end{align} 
where $f : \R^n \to \R$ is the unknown objective function, $\eta_t > 0$ is the step size (learning rate), and $ \wh{\nabla} f_{\V_t}^{\delta_t} (\x_t) $ is the estimator of $\nabla f $ at $\x_t$ defined as follows. 
\begin{align} 
    \wh{\nabla} f_{\V_t}^{\delta_t} (\x) := \frac{n}{2 \delta_t k} \sum_{i=1}^k \( f ( \x + \delta_t \v_{t,i} ) - f ( \x - \delta_t \v_{t,i}) \) \v_{t,i}, \quad \forall \x \in \R^n, \label{eq:def-grad-est-intro} 
\end{align} 
where $ \V_t = [ \v_{t,1}, \v_{t,2}, \cdots, \v_{t,k}] $ is uniformly sampled from the Stiefel manifold $ \text{St} (n,k) := \{ \mathbf{X} \in \R^{n \times k} : \mathbf{X}^\top \mathbf{X} = \mathbf{I}_k \} $, and $\delta_t := 2^{-t}$ is the finite difference granularity. Previously, the statistical properties of (\ref{eq:def-grad-est-intro}) have been investigated by \cite{fengwang2022} (See Section \ref{sec:prelim}). 
Throughout, we use Stochastic Zeroth-order Gradient Descent (SZGD) to refer to the above update rule (\ref{eq:gd}) with the estimator (\ref{eq:def-grad-est-intro}).

%
%
%

In this paper, we prove the following results for SZGD. Let the objective function $f$ 
satisfy the \Loj inequality with exponent $\theta$ (Definition \ref{def:loj}). Let $\{ \x_t \}_t$ be the sequence generated by SZGD. Then under Assumption \ref{assumption:smooth}, 
\begin{itemize}[align=left,leftmargin=*]  
    \item The sequence $ \{ \x_t \}_t $ converges to a limit $\x_\infty$ almost surely. In addition, $\x_\infty$ is a critical point of $f$. 
    \item If $\theta \in ( 0,\frac{1}{2} ]$, then there exists $Q > 1$ such that $ \{ Q^t ( f (\x_t) - f (\x_\infty) ) \}_t $ converges to zero in expectation. In other words, if $\theta \in (0,\frac{1}{2}]$, $ \{f (\x_t) - f (\x_\infty) \}_t $ converges to zero linearly in expectation. 
    \item If $\theta \in (\frac{1}{2}, 1)$, $\{ f (\x_t) - f (\x_\infty) \}_t$ converges to zero at rate $ O \( t^{ \frac{1 }{ 1 - 2\theta } } \) $ in expectation. 
    This rate is faster than the convergence rate of $\{ \| \x_t - \x_\infty \| \}_t$ previously obtained in \cite{attouch2009convergence}, where $ \{ \x_t \}_t $ is generated by the proximal algorithm.  
\end{itemize}

Also, we prove the following convergence rate for $\{ \| \x_t - \x_\infty \| \}_t$ and $\{ \sum_{s=t}^\infty \| \x_{s+1} - \x_{s} \|^2 \}_t$. 
\begin{itemize}[align=left,leftmargin=*] 
    \item If $\theta \in ( 0,\frac{1}{2} ]$, then there exists $Q > 1$ such that $ \{ Q^t \sum_{s=t}^\infty \| \x_{s+1} - \x_{s} \|^2 \}_t $ and $ \{ Q^t \| \x_{t} - \x_{\infty} \| \}_t $ converges to zero in expectation. In other words, if $\theta \in ( 0,\frac{1}{2} ]$, $ \{ \sum_{s=t}^\infty \| \x_{s+1} - \x_{s} \|^2 \}_t $ and $ \{ \| \x_{t} - \x_{\infty} \| \}_t $ converges to zero linearly in expectation. 
    \item If $\theta \in (\frac{1}{2}, 1)$, then $ \{ \sum_{s=t}^\infty \| \x_{s+1} - \x_{s} \|^2 \}_t $ converges to zero at rate $ O \( t^{\frac{1}{ 1 - 2\theta}} \) $ in expectation, and $ \{ \| \x_{t} - \x_{\infty} \| \}_t $ converges to zero at rate $ O \( t^{\frac{1-\theta}{ 2\theta - 1}} \) $ in expectation. 
\end{itemize} 

For $ \{ \| \x_t - \x_\infty \| \}_t $, the convergence rate of SZGD matches the convergence rate of the proximal algorithm \cite{attouch2009convergence}, for any $\theta \in (\frac{1}{2},1)$. This means one does not need a proximal oracle to obtain the state-of-the-art of convergence rate for $ \{ \| \x_t - \x_\infty \| \}_t $. 
More importantly, the above results imply that for any fixed $ \theta \in (\frac{1}{2}, 1) $, $ \{ f (\x_t) - f (\x_\infty) \}_t $ can converge much faster than $ \{ \| \x_t - \x_\infty \| \}_t $. 
Also, we show that the Gradient Descent (GD) algorithm converges at the same rate as SZGD. 

\subsubsection*{Convergence Rates of the Proximal Algorithm} 

We also provide convergence rates of $ \{ f (\x_t) - f (\x_\infty) \}_t$ for the proximal algorithm. Let $ \{ \x_t \}_t $ be the sequence generated by the proximal algorithm. In their classic work \cite{attouch2009convergence}, Attouch and Bolte showed that, when $ \theta \in (\frac{1}{2}, 1), $ $ \{ \| \x_t - \x_\infty \| \}_t$ converges to zero at rate $ O \( t^{\frac{1-\theta}{1-2\theta} } \) $.  
In this paper, we also prove the following result. 
\begin{itemize}[align=left,leftmargin=*]  
    \item Let $\{ \x_t \}_t$ be generated by the proximal algorithm. If $\theta \in (\frac{1}{2}, 1)$ and $f$ is convex, then $\{ f (\x_t) - f (\x_\infty) \}_t$ converges to zero at rate $ O \( t^{ \frac{1 }{ 1 - 2\theta } } \) $. Note that this rate can be much faster than the convergence rate of $ \{ \| \x_t - \x_\infty \| \}_t $ previously obtained in \cite{attouch2009convergence}, which is of order $ O \( t^{\frac{1-\theta}{1-2\theta} } \) $. 
\end{itemize} 

\begin{remark}[Summary of contributions]
    We prove convergence rates of SZGD for \L ojasiewicz Functions. 
    Our results generalize the recent work that studies SZGD for Polyak--\Loj functions \cite{kozak2022zeroth}. 
    
    More importantly, our results show that, when the \Loj exponent $\theta \in (\frac{1}{2}, 1)$, $ \{ f (\x_t) - f (\x_\infty) \}_t $ can converge much faster than $ \{ \| \x_t - \x_\infty \| \}_t $. This observation is true for both SZGD and the proximal algorithm. Recently, \cite{https://doi.org/10.1112/blms.12586} constructs a convex \Loj function whose gradient flow is spiral, and fails the Thom's gradient conjecture. Our results show that although the gradient flow may be spiral, the function value converges relatively fast. Thus the results in  \cite{https://doi.org/10.1112/blms.12586} is not too pessimistic. 

\end{remark}

\section{Related Works}
\label{sec:related-works} 


Zeroth order optimization is a central scheme in many fields (e.g., \cite{nelder1965simplex,conn2009introduction,shahriari2015taking}). Among many zeroth order optimization mechanisms, a classic and prosperous line of works focuses on estimating gradient/Hessian using zeroth order information and use the estimated gradient/Hessian for downstream optimization algorithms. 






A classic line of related works is the Robbins--Monro--Kiefer--Wolfowitz-type algorithms \cite{robbins1951stochastic,kiefer1952stochastic} from stochastic approximation. See (e.g., \cite{kushner2003stochastic,benveniste2012adaptive}) for exposition. The Robbins–Monro and Kiefer–Wolfowitz scheme has been used in stochastic optimization and related fields (e.g., \cite{nemirovski2009robust,wang2017stochastic}). In particular, \cite{bertsekas2000gradient} have shown that stochastic gradient descent algorithm either converges to a stationary point or goes to infinity, almost surely. While the results of \cite{bertsekas2000gradient} is quite general, no convergence {rate} is given. As an example of recent development, \cite{wang2017stochastic} showed that convergence results for Robbins--Monro when the objective is nonconvex. While the study of Robbins--Monro and Kiefer--Wolfowitz has spanned 70 years, stochastic zeroth order optimization has not come to its modern form until early this century. 

In recent decades, due to lack of direct access to gradients in real-world applications, zeroth order optimization has attracted the attention of many researchers. In particular, \cite{flaxman2005online} introduced the single-point gradient estimator for the purpose of bandit learning. Afterwards, many modern gradient/Hessian estimators have been introduced and subsequent zeroth order optimization algorithms have been studied. To name a few, \cite{duchi2015optimal,nesterov2017random} have studied zeroth order optimization algorithm for convex objective and established in expectation convergence rates. \cite{balasubramanian2021zeroth} used the Stein's identity for Hessian estimators and combined this estimator with cubic regularized Newton's method \cite{nesterov2006cubic}.  \cite{wang2021GW,wang2022hess} provided refined analysis of Hessian/gradient estimators over Riemannian manifolds. \cite{li2022stochastic} studied zeroth order optimization over Riemannian manifolds and proved in expectation convergence rates. The above mentioned stochastic zeroth order optimization works focus {in expectation} convergence rates. Probabilistically stronger results have also been established for stochastic optimization methods recently. For example, \cite{lan2012optimal} provide high probability convergence rates for composite optimization problems. 

The study of \Loj functions forms an important cluster of related works. \Loj functions satisfies the \Loj inequality with \Loj exponent $\theta$ \cite{loj1963}. An important special case of the \Loj inequality is the Polyak--\Loj inequality \cite{polyak1963}, which corresponds to the \Loj inequality with $\theta = \frac{1}{2}$. In \cite{klz1995,Kurdyka1998OnGO}, the \Loj inequality was generalized to the Kurdyka--\Loj inequality. Subsequently, the geometric properties has been intensively studied, along with convergence studies of optimization algorithms on Kurdyka--\Loj-type functions \cite{polyak1963,loj1963,klz1995,Kurdyka1998OnGO,kurdyka2000proof,bolte2007lojasiewicz,attouch2009convergence,li2018calculus}. Yet no prior works focus on stochastic zeroth order methods for \Loj functions. 



Perhaps the single most related work is the recent work by \cite{kozak2022zeroth}. In \cite{kozak2022zeroth}, in expectation convergence rates are proved for Polyak--\Loj functions. 
Compared to \cite{kozak2022zeroth}'s study of Polyak--\Loj functions, our results are more general since we \Loj functions are more general than Polyak--\Loj functions. 



\section{Preliminaries} 
\label{sec:prelim} 

\subsection{Gradient Estimation} 

Consider gradient estimation tasks in $\R^n$. 
The gradient estimator we use is \cite{fengwang2022}: 
\begin{align} 
    \wh{\nabla} f_{\V}^\delta (\x) := \frac{n}{2 \delta k} \sum_{i=1}^k \( f ( \x + \delta \v_{i} ) - f ( \x - \delta \v_{i}) \) \v_{i}, \quad \forall \x \in \R^n, 
    \label{eq:def-grad-est}
\end{align} 
where $ [\v_{1}, \v_{2}, \cdots, \v_{k}] = \V $ is uniformly sampled from the Stiefel manifold $ \text{St} (n,k) = \{ \mathbf{X} \in \R^{n \times k} : \mathbf{X}^\top \mathbf{X} = \mathbf{I}_k \} $, and $\delta$ is the finite difference granularity. In practice, one can firstly generate a random matrix $\mathbf{U} \in \R^{n \times k}$ of $i.i.d.$ standard Gaussian ensemble. Then apply the Gram--Schmit process on $\mathbf{U}$ to obtain the matrix $\V$. 

\begin{definition} 
    \label{def:smooth} 
    A function $f : \R^n \rightarrow \R$ is called $ L $-smooth if it is continuously differentiable, and 
        $\| \nabla f (\x) - \nabla f (\x') \| \le L \| \x - \x' \|, $ for all $\x,\x' \in \R^n. $
\end{definition}

With the above description of smoothness, we can state theorem on statistical properties of the estimator (\ref{eq:def-grad-est}). These properties are in Theorems \ref{thm:bias} and \ref{thm:variance}.  


\begin{theorem}[\cite{flaxman2005online}] 
    \label{thm:bias} 
    If $ f $ is $L$-smooth, then the gradient estimator $ \wh{\nabla} f_{\V}^{\delta} $ satisfies 
        $ \left\| \E \[ \wh{\nabla}  f_{\V}^{\delta} (\x)  \] - \nabla f (\x) \right\| \le \frac{L n  \delta}{n+1} $ for all $ \x \in \R^n $. 
\end{theorem} 



\begin{theorem} 
    \label{thm:variance} 
    If $f$ is $L$-smooth, then variance of the gradient estimator for $f$ (Eq. \ref{eq:def-grad-est}) satisfies 
    \begin{align*}
        &\; \E \[ \left\| \wh{\nabla}f_{\V}^\delta (\x) - \E \[ \wh{\nabla}f_{\V}^\delta (\x) \] \right\|^2 \] \\ 
        \le& \; 
        \( \frac{n}{k} - 1 \) \| \nabla f (\x) \|^2 + \frac{ 4 L \delta}{ \sqrt{3} } \( \frac{n^2}{k}  - n \) \| \nabla f (\x) \| + \frac{ 4 L^2 n^2 \delta^2 }{ 3 k } , 
    \end{align*}
    for all $x \in \R^n$. 
    
\end{theorem}



The proof of Theorem \ref{thm:variance} can be found in Section \ref{sec:proof}. 






    

\subsection{\Loj Functions} 

\Loj functions are functions that satisfies the \Loj inequality. We start with the differentiable \Loj functions (Definition \ref{def:loj}). A more general version of the \Loj inequality \cite{loj1963}, where gradient is replaced by subgradients, is discussed in Section \ref{sec:nonsmooth}. 

\begin{definition}
    \label{def:loj} 
    A differentiable function is said to be a (differentiable) \L ojasiewicz function with \L ojasiewicz exponent $\theta \in (0,1)$ if 
    for any $\x^*$ with $\nabla f (\x^*) = 0$, there exist constants $ \kappa, \mu > 0 $ such that  
    \begin{align} 
        | f (\x) - f (\x^*) |^\theta \le \kappa \| \nabla f (\x) \|, \quad \forall \x \text{ with $ \| \x - \x^* \| \le \mu $}. \label{eq:loj-ineq}
    \end{align} 
    We call (\ref{eq:loj-ineq}) \Loj inequality. Without loss of generality, we let $\kappa = 1$ to avoid clutter. 
\end{definition} 

An important special case of the \Loj inequality is the Polyak--\Loj inequality \cite{polyak1963}, which corresponds to a case of the \Loj exponent $\theta$ being $\frac{1}{2}$. Also, \Loj inequality is an important special case of the Kurdyka--\Loj (KL) inequality \cite{klz1995,Kurdyka1998OnGO}. Roughly speaking, the KL inequality does not assume differentiability, and replaces the left-hand-side of (\ref{eq:loj-ineq}) with a more general function (in $ f(\x) $). 
In its general form, the \Loj inequality does not require the objective function to be differentiable. In such cases, gradient on the right-hand-side of (\ref{eq:loj-ineq}) is replaced with subgradient. 
We will discuss the subgradient version in Section \ref{sec:nonsmooth}.

\subsection{Conventions and Notations}

Before proceeding to the main results, we put forward several conventions. 
\begin{itemize}[align=left,leftmargin=*] 
    \item We use lower case bold letters (e.g. $\x_t, \u$) to refer to vectors, upper case bold letters (e.g., $\P_t$, $\V_t$) to refer to matrices. 
    \item We use $C$ to denote non-stochastic constants that is independent of $t$, not necessarily referring to the same value at each occurrence. Such constants $C$ always appear in front $\delta_t$ \if\submission1\textcolor{red}{(or other terms that are exponentially small in $t$)}\else(or other terms that are exponentially small in $t$)\fi. Since terms like $ \delta_t $ decrease to zero exponentially fast, we use such $C$ to avoid notational clutter in front of exponentially small terms. 
    \item 
    For any random variable (or collection of random variables) $ X $, we use $\E_{X} \[ \cdot \]$ to denote the expectation with respect to $X$. We use $\F_t$ to denote the $\sigma$-algebra generated by all randomness after arriving at $\x_t$, but before obtaining the estimator $\wh{\nabla} f_{\V_t}^{\delta_t} (\x_t)$, and use $\E_t [\cdot]$ to denote the expectation condition on $\mathcal{F}_t$. Also, we use $\E \[ \cdot\]$ to denote the total expectation. 
\end{itemize} 

For Section \ref{sec:conv}, the objective function $f$ satisfies Assumption \ref{assumption:smooth}. 

\begin{assumption}
    \label{assumption:smooth} 
    Throughout Section \ref{sec:conv}, the objective function $f$ satisfies: 
    \begin{enumerate}[label=(\textit{\roman*})]
        \item $ f $ is $L$-smooth for some constant $L > 0$. (See Definition \ref{def:smooth}). 
        \item $f$ is a (differentiable) \Loj function with \Loj exponent $\theta$. (See Definition \ref{def:loj}) 
        \item $\inf_{x \in \R^n} f (\x) > -\infty$. 
        \item All critical points of $f$ are isolated points (of $\R^n$). 
        \item Let $ \{ \x_t \}_t $ be the sequence generated by the SZGD algorithm. We assume that there exists a critical point $x^*$ such that $ \| \x_t - \x^* \| \le \mu $ for all $t$. 
        
    \end{enumerate} 
\end{assumption} 



Many items in the above assumptions are assumed in the classic work \cite{attouch2009convergence}. See Assumption \ref{assumption:nonsmooth} for the set of assumptions previously employed in \cite{attouch2009convergence}. 

A final remark before proceeding to the main results is that we focus on stochastic zeroth-order optimization with noiseless function evaluations. The algorithm is random, and the environment is noiseless.





\section{Convergence Analysis for SZGD} 
\label{sec:conv}

This section presents convergence analysis for the SZGD algorithm. Before proceeding, we first summarize SZGD in Algorithm \ref{alg}. 

\begin{algorithm}[h] 
	\caption{Stochastic Zeroth-order Gradient Descent (SZGD)} 
	\label{alg} 
	\begin{algorithmic}[1] 
	    \STATE \textbf{Input:} Dimension $n$; Number of orthogonal directions for the estimators $k$. 
		/* function $f$ is $L$-smooth. */      
		\STATE Pick $\x_0 \in \R^n$ and $\delta_0 = 1$. /* or $\delta_0 \in (0,1)$. */
		\STATE Pick step size lower and upper bounds $\eta_-, \eta_+ \in (0,\infty)$. 
		\FOR{$t = 0, 1,2,\cdots$} 
                \STATE Sample $\V_t \sim \text{Unif} (\St (n,k))$, and use the random directions in $ \V_t $ to define $\wh{\nabla} f_{\V_t}^{\delta_t} (\x_t)$. 
		    \STATE $\x_{t+1} = \x_t - \eta_t \wh{\nabla} f_{\V_t}^{\delta_t} (\x_t)$. 
		    \STATE $\delta_{t+1} = \delta_t / 2$. 
		\ENDFOR 
	\end{algorithmic} 
\end{algorithm} 


The main convergence rate guarantee for SZGD is in Theorem \ref{thm:rate}. 

\begin{theorem} 
    \label{thm:rate}
    Instate Assumption \ref{assumption:smooth}. Let $\x_t$ be a sequence generated by SZGD (Algorithm \ref{alg}). 
    Pick step sizes $\eta_t$ so that there exist $ \eta_-, \eta_+ \in (0, \infty) $ such that $ \eta_- \le \eta_t \le \eta_+ $ for all $t$. 
    Then $\{ \x_t \}$ converges to a critical point $\x_\infty$ almost surely. In addition, it holds that 
    \begin{enumerate}[label=(\alph*)] 
        \item if $ \theta \in (0, \frac{1}{2} ] $, $ - 1 \le \( \frac{L n \eta_-^2 }{2k} - \eta_- \) < 0 $ and $ - 1 \le \( \frac{L n \eta_+^2 }{2k} - \eta_+ \) < 0 $, then there exists a constant $Q > 1$ such that $ \{ Q^t \( f (\x_t) - f (\x_\infty) \) \}_t $ converges to $0$ in expectation. 
        \item if $ \theta \in ( \frac{1}{2}, 1 ) $ and $ \eta_-, \eta_+ \in \( 0, \frac{2k}{Ln} \) $, then it holds that, 
        $\{ f (\x_t) - f (\x_\infty)  \}_t$ converges to zero at rate $ O \( t^{ \frac{1  }{ 1- 2 \theta } } \) $ in expectation. 
    \end{enumerate}
\end{theorem} 

This theorem gives convergence rate of $ \{ f (\x_t) \}_{t \in \mathbb{N}}$. 
Similar convergence guarantees for $\{ \x_t \}_t$ can be found in Section \ref{sec:conv-rate-x}.



The rest of this section is organized as follows. 
In Section \ref{sec:asymp}, we show that Algorithm \ref{alg} converges almost surely. In particular, we show that the sequences $\{ \x_t \}_{t \in \mathbb{N}}$ generated by SZGD converges to a critical point almost surely. 
In Section \ref{sec:conv-rate}, we establish convergence rates of $\{ f (\x_t) \}_t$ for Algorithm \ref{alg}. Then in Section \ref{sec:conv-rate}, we establish convergence rates of $\{ \x_t \}_t$ for Algorithm \ref{alg}. 

\subsection{Asymptotic Convergence}
\label{sec:asymp} 

We will first show that $\{ \x_t \}_t$ converges to a critical point almost surely. 
For simplicity, let 
\begin{align} 
    - B := \max \left\{ \( \frac{ L \eta_-^2 n }{ 2k } - \eta_- \), \( \frac{ L \eta_+^2 n }{ 2k } - \eta_+ \) \right\}. \label{eq:def-B} 
\end{align}

\begin{proposition} 
    \label{prop:rv-smooth}
    Let $ \P_t := \V_t \V_t^\top $. Then it holds that 
    \begin{align*} 
        f (\x_{t+1}) 
        \le 
        f (\x_t) - B \frac{n}{k} \nabla f (\x_t)^\top \P_t \nabla f (\x_t) + C \delta_t . 
    \end{align*} 
    In addition, we have 
        $
        \E_t \[ f (\x_{t+1}) \] 
        \le 
        f (\x_t) - B \| \nabla f (\x_t) \|^2 + C \delta_t . 
        $ 
\end{proposition} 

\begin{proof} 
    Since $ f (\x)$ is $L$-smooth ($\nabla f (\x)$ is $L$-Lipschitz), $\nabla^2 f (\x) $ (the weak total derivative of $ \nabla f (\x) $) is integrable. Let $\v \in \R^n$ be an arbitrary unit vector. When restricted to any line along direction $ \v\in \R^n $, it holds that $ \v^\top \nabla^2 f (\x) \v $ (the weak derivative of $ \v^\top \nabla f (\x) $ along direction $\v$) has bounded $L_\infty$-norm. This is due to the fact that Lipschitz functions on any closed interval $[a,b]$ forms the Sobolev space $ W^{1,\infty} [a,b] $. 
    
    Next we look at the variance bound for the estimator. Without loss of generality, we let $\x = \0$. Bounds for other values of $\x$ can be similarly obtained. 
    
    Taylor's expansion of $ f $ with integral form gives 
    \begin{align*} 
        f (\delta \v) 
        =& \; 
        f (\0) + \delta \v^\top \nabla f (\0) + \int_{0}^{\delta} (\delta - t) \v^\top \nabla^2 f ( t \v ) \v \, dt 
    \end{align*} 
    
    Thus for any $\v \in \S^{n-1}$ and small $\delta$, 
    \begin{align*} 
        &\; \frac{1}{2} \big( f (\delta \v) - f (-\delta \v) \big) \\
        =& \;  
        \delta \v^\top \nabla f (\0) + \frac{1}{2} \int_{0}^{\delta} (\delta - t) \v^\top \nabla^2 f ( t \v ) \v \, dt - \frac{1}{2} \int_{0}^{-\delta} (-\delta - t) \v^\top \nabla^2 f ( t \v ) \v \, dt . 
    \end{align*} 

    Therefore, 
    \begin{align*} 
        \wh{\nabla} f_{\V_t}^{\delta_t} (\x_t) 
        =& \; 
        \frac{n }{2 k \delta_t } \sum_{i=1}^k \( f (\x_t + \delta_t \v_{t,i}) - f (\x_t - \delta_t \v_{t,i}) \) \v_{t,i} \\ 
        =& \; 
        \frac{n}{k} \sum_{i=1}^k \v_{t,i} \v_{t,i}^\top \nabla f ( \x_t ) + \frac{n}{2k \delta_t } \sum_{i=1}^k \v_{t,i} \int_{0}^{\delta_t} ( \delta_t - t ) \v_{t,i}^\top \nabla^2 f ( \x_t + t \v_{t,i} ) \v_{t,i} \, dt \\ 
        &- \frac{n}{2k \delta_t} \sum_{i=1}^k \v_{t,i} \int_{0}^{-\delta_t} (-\delta_t - t) \v_{t,i}^\top \nabla^2 f ( \x_t + t \v_{t,i} ) \v_{t,i} \, dt , 
    \end{align*} 
    \if\submission1
    \emph{\textcolor{red}{(Compared to the previous version, one line of equation is removed.)}}\fi 
    
    which implies 
    \begin{align*} 
        \nabla f (\x_t)^\top \wh{\nabla} f_{\V_t}^{\delta_t} (\x_t) 
        =& \;  
        \frac{n}{k} \nabla f (\x_t)^\top \P_t \nabla f (\x_t) + O ( L n \| \nabla f (\x_t) \| \delta_t ) , \\ 
        \left\| \wh{\nabla} f_{\V_t}^{\delta_t} (\x_t) \right\|^2 
        =& \;  
        \if\submission1{\textcolor{red}{\frac{n^2}{k^2}}}\else\frac{n^2}{k^2}\fi \nabla f (\x_t)^\top \P_t \nabla f (\x_t) + O ( L \if\submission1{\textcolor{red}{n^2}}\else n^2\fi \( \| \nabla f (\x_t) \| + \| \nabla f (\x_t) \|^2 \) \delta_t ) . 
    \end{align*} 

    Since $ \{ \x_t \}_t$ is bounded and $ \| \nabla f ( \x_t) \|  $ is continuous (Assumption \ref{assumption:nonsmooth}), we know $ O ( L n \| \nabla f (\x_t) \| \delta_t ) \le C \delta_t $ and $ O ( L \if\submission1{\textcolor{red}{n^2}}\else{n^2}\fi \( \| \nabla f (\x_t) \| + \| \nabla f (\x_t) \|^2 \) \delta_t ) \le C \delta_t $. 
    
    Thus by $L$-smoothness of the $ f $, we have 
    \begin{align*} 
        f (\x_{t+1}) 
        \le& \; 
        f (\x_t) + \nabla f (\x_t)^\top \( \x_{t+1} - \x_t \) + \frac{L}{2} \| \x_{t+1} - \x_t \|^2 \\ 
        =& \; 
        f (\x_t) - \eta_t \nabla f (\x_t)^\top \wh{\nabla} f_{\V_t}^{\delta_t} (\x_t) + \frac{L \eta_t^2 }{2} \| \wh{\nabla} f_{\V_t}^{\delta_t} (\x_t) \|^2 \\ 
        \le& \; 
        f (\x_t) - B \frac{n}{k} \nabla f (\x_t)^\top \P_t \nabla f (\x_t) + C \delta_t . 
    \end{align*} 

    Since $\F_t$ contains $\x_t$ but not $\V_t$, it holds that 
    \begin{align*} 
        \E_t \[ f (\x_{t+1}) \] 
        \le 
        f (\x_t) - B \frac{n}{k} \nabla f (\x_t)^\top \E_t \[ \P_t \] \nabla f (\x_t) + C \delta_t . 
    \end{align*} 

    By Propositions \ref{prop:uniform} and \ref{prop:sphere-exp}, we know $ \E_t \[ \P_t \] = \frac{k}{n} \mathbf{I} $. Combining this fact with the above equation finishes the proof. 
\end{proof} 

In Lemma \ref{lem:convergence-grad}, we show that $ \left\{ \| \nabla f (\x_t) \| \right\}_{t \in \mathbb{N}} $ converges to zero. 

\begin{lemma} 
    \label{lem:convergence-grad} 
    Instate Assumption \ref{assumption:smooth}. Let $\{ \x_t \}_t$ be the sequence governed by Algorithm \ref{alg}. Pick step sizes $\eta_t$ so that there exist $ \eta_-, \eta_+ \in (0, \infty) $ such that $ \eta_- \le \eta_t \le \eta_+ $ for all $t$. Then $ \left\{ \| \nabla f (\x_t) \| \right\}_{t \in \mathbb{N}} $ converges to zero almost surely. 
\end{lemma} 

\begin{proof} 

    By Proposition \ref{prop:rv-smooth}, we know 
    \begin{align} 
        f ( \x_{t+1}) 
        \le&\; 
        f ( \x_t) - \frac{n}{k} B \nabla f ( \x_t)^\top \P_t \nabla f ( \x_t) + C \delta_t . \label{eq:for-convergence} 
    \end{align} 

    Taking conditional expectation on both sides of the above inequality gives 
    \begin{align*} 
        \E \[ f (\x_{t+1}) - f (\x_t) \] 
        \le &\; 
        - B \E \[ \frac{n}{k} \nabla f (\x_t)^\top \E_t \[ \P_t \] \nabla f (\x_t) \] + C \delta_t \\ 
        =& \; 
        - B \E \[ \| \nabla f (\x_t) \|^2 \] + C \delta_t . 
    \end{align*} 
    
    Suppose, in order to get a contradiction, that there exists $ \alpha > 0 $ such that $ \| \nabla f (\x_t) \|^2 > \alpha $ infinitely often. Thus we have 
    \begin{align*} 
        \E \[ f (\x_{t}) \] 
        \le 
        \E \[ f (\x_{0}) \] -  B \E \[ \sum_{s\in\mathcal{X}_t} \| \nabla f (\x_s) \|^2 \] + C \delta_t , \quad t \ge T_0, 
    \end{align*} 
    where $\mathcal{X}_t=\{0\leq s <t:\;\|\nabla f(\x_s)\|^2>\alpha\}$.
    The above inequality (\ref{eq:for-convergence}) gives
    \begin{align*}
        \E \[ f (\x_{t}) \] 
        \le 
        \E \[ f (\x_{0}) \] - \alpha B N_t + C \delta_t, \quad t \ge T_0
    \end{align*}
    for some $ \{ N_t \}_t \subseteq \mathbb{N} $ such that $\lim\limits_{ t \to \infty} N_t = \infty$. 
    
    Taking limits on both sides of the above inequality gives  
    \begin{align*} 
        \liminf_{t \to \infty } \E \[ f (\x_{t}) \] 
        \le 
        \E \[ f (\x_{0}) \] + \liminf_{t \to \infty } ( {\alpha} B N_t ) + C \delta_t = -\infty, 
    \end{align*} 
    which leads to a contradiction to $ \inf_{\x\in\R^n} f(\x) > -\infty $ (Assumption \ref{assumption:smooth}). 
    
    By the above proof-by-contradiction argument, we have shown
    \begin{align*} 
        \Pr \( \| \nabla f (\x_t) \|^2 > \alpha \text{ infinitely often} \) = 0, \qquad \forall \alpha > 0. 
    \end{align*}  
    
    Note that  
        $ \left\{ \lim_{t \rightarrow \infty } \| \nabla f (\x_t) \|^2 = 0 \right\}^c 
        =
        \bigcup_{\alpha > 0} \{ \| \nabla f (\x_t) \|^2 > \alpha \text{ infinitely often} \} 
        =
        \bigcup_{\alpha \in \mathbb{Q}_+ } \{ \| \nabla f (\x_t) \|^2 > \alpha \text{ infinitely often} \} , $
    where we replace the union over an uncountable set with a union over a countable set. 
    
    Thus it holds that 
    \begin{align*} 
        & \; 1 - \Pr \( \lim_{t \rightarrow \infty} \| \nabla f (\x_t) \|^2 = 0 \) 
        =  
        \Pr \( \bigcup_{ \alpha \in \mathbb{Q}_+ } \{ \| \nabla f (\x_t) \|^2 > \alpha \text{ infinitely often} \}  \) \\
        \le& \; 
        \sum_{b \in \mathbb{Q}_+} \Pr \( \| \nabla f (\x_t) \|^2 > \alpha \text{ infinitely often} \)
        =
        0 , 
    \end{align*} 
    which concludes the proof. 
    
\end{proof} 

As consequences of Lemma \ref{lem:convergence-grad}, we know that $ \{ \x_t \}_{t \in \mathbb{N}} $ and $ \{ f (\x_t) \}_{t \in \mathbb{N}} $ converge almost surely. 

\begin{lemma}
    \label{lem:convergence-pt}
    Instate Assumption \ref{assumption:smooth}. Let $\{ \x_t \}_{t \in \mathbb{N}}$ be the sequence of Algorithm \ref{alg}. Pick step sizes $\eta_t$ so that there exist $ \eta_-, \eta_+ \in (0, \infty) $ such that $ \eta_- \le \eta_t \le \eta_+ $ for all $t$. Then $\{ \x_t \}_t$ converges to a bounded critical point of $ f $ almost surely. 
    Let $ \x_\infty $ be the almost sure limit of $ \{ \x_t \}_{t \in \mathbb{N}} $. It holds that $ \lim\limits_{ t \to \infty} f (\x_t) = f (\x_\infty) $ almost surely. 
\end{lemma} 


\begin{proof}[Proof of Lemma \ref{lem:convergence-pt}] 

    Let $ \V_t \in \St (n,k) $ be the random orthogonal frame at $t$. 
    \begin{quote} 
        Let $ R $ be an arbitrary realization of $ \V_1, \V_2, \cdots $ such that $ \left\{ \| {\nabla} f (\x_t) \| \right\}_{t \in \mathbb{N}} $ converges to zero. \textbf{(H0)} 
    \end{quote} 
    
    Next we restrict our attention to this realization $R$. 
    Note that there is no randomness in $ \{ \x_t \}_{t \in \mathbb{N}}$ once this realization $R$ is fixed. 
    
    Let $ \u \in \S^{n-1}  $ be an arbitrary unit vector, and let $z_t := \u^\top \x_t$ and $ \varphi_t := \eta_t \u^\top \wh{\nabla} f_k^{\delta_t} (\x_t) $ for all $t$. With this notation, we have $ z_{t+1} = z_t - \varphi_t $ for all $t$. By Lemma \ref{lem:convergence-grad}, we know $\lim_{t \rightarrow \infty} \( z_{t+1} - z_t \) = 0$. 
    Let $ K_{\u} $ be the set of subsequential limits of $\{ z_t \}_t$. 
    Then $K_{\u}$ is closed and bounded (by Assumption \ref{assumption:smooth}). 

    \begin{claim}
        \label{clm:connect}
        The set $K_{\u}$ is connected for any unit vector $\u \in \R^n$. 
    \end{claim}

    \begin{proof}[Proof of Claim \ref{clm:connect}] 


        Suppose, in order to get a contradiction, that $ K_{\u} $ is not connected. \textbf{(H)}
        
        Since $ K_{\u} $ is closed and bounded, we know that, if $ K_{\u} $ is disconnected, then there exists $ a_1, b_1, a_2, b_2 \in (-\infty, \infty) $ such that
        \begin{itemize}
            \item $ [a_1,b_1] \cup [a_2, b_2] \subseteq K_{\u} $ with $ b_1 < a_2 $; 
            \item $ (b_1, a_2) \cap K_{\u} = \emptyset $. 
        \end{itemize}
        
        For simplicity, let $ m := \frac{b_1 + a_2}{2} $ and $ \Delta := a_2 - b_1 $. Since $ b_1 $ and $ a_2 $ are limit points, we can find subsequences $ \{ z_{i_j^1} \}_j $ and $ \{ z_{i_j^2} \}_j $ such that 
        \begin{enumerate} 
            \item $ z_{i_j^1} \le m - \frac{\Delta}{4} $; 
            \item $ z_{i_j^2} \ge m + \frac{\Delta}{4} $; 
            \item $ i_j^1 < i_j^2 $ for all $j$. 
        \end{enumerate} 

        Since $ \{ \x_{t+1} - \x_t \}_t $ converges to zero, we know that $ \{z_{t+1} - z_t\}_t$ converges to zero. By items 1, 2 and 3 above, we know that there exists $ \{ i_j^3 \}_j $ such that $ i_j^1 < i_{j}^3 \le i_j^2 $ and $z_{i_j^3} \in ( m - \frac{\Delta}{4}, m + \frac{\Delta}{4} ] $ for infinitely many $j$. Otherwise, there will be a contradiction to the fact that $ \{z_{t+1} - z_t\}_t$ converges to zero. Then we know that $ \{ z_t \}_t $ has a limit point in $ [ m - \frac{\Delta}{4}, m + \frac{\Delta}{4} ] $, since $ \{ z_{i_j^3} \}_j $ has a limit point in $ [ m - \frac{\Delta}{4}, m + \frac{\Delta}{4} ] $. This is a contradiction to \textbf{(H)}. 
        
    \end{proof} 

    By Claim \ref{clm:connect}, we know that the limit points of $\{ \x_t \}_{t \in \mathbb{N}}$ is connected. In this realization $R$, it holds that that any limit point of $\{ \x_t \}_{t \in \mathbb{N}}$ is a critical point of $f$. By Assumption \ref{assumption:nonsmooth}, critical points of $f$ are isolated. Thus we know that the limit points of $\{ \x_t \}_{t \in \mathbb{N}}$ must be a singleton. In other words, the sequence $\{ \x_t \}_{t \in \mathbb{N}}$ converges to a critical point of $f$. 

    The above argument holds for an arbitrary realization $R$ that satisfies \textbf{(H0)}. By Lemma \ref{lem:convergence-grad}, almost all realizations satisfy \textbf{(H0)}. This finishes the proof. 
    
    \if\submission1
    \hfill $\blacksquare$ 
    \fi
    
\end{proof}  


Now that we have shown that $\{ \x_t \}_t$ converges to a critical point almost surely. We state a more convenient form of the \Loj inequality in the following proposition. 

\begin{proposition} 
    \label{prop:conv-loj} 
    Instate Assumption \ref{assumption:smooth}. Pick step sizes $\eta_t$ so that there exist $ \eta_-, \eta_+ \in (0, \infty) $ such that $ \eta_- \le \eta_t \le \eta_+ $ for all $t$. Let $ \x_\infty $ be the almost sure limit of $ \{ \x_t \}_{t \in \mathbb{N}} $. Then it holds that, for all $ t \ge T_0 $,  
    \begin{align*} 
        ( f (\x_t ) - f (\x_\infty) )^{2\theta} \le \| \nabla f (\x_t) \|^{\if\submission1{\textcolor{red}{2}}\else{2}\fi} , \quad \text{almost surely.} 
    \end{align*} 
\end{proposition} 

\begin{proof} 
    Proposition \ref{prop:conv-loj} is a direct consequence of $a.s.$ convergence of $\{ \x_t \}_t$ (Lemma \ref{lem:convergence-pt}) and the \Loj inequality. 
\end{proof}


\subsection{Convergence Rate of $ \{ f (\x_t) \}_{t \in \mathbb{N}}$} 
\label{sec:conv-rate} 

In the previous subsection, we have proved asymptotic convergence results for the SZGD algorithm. This section is devoted to convergence rate analysis of $ \{ f (\x_t) \}_{t \in \mathbb{N}}$. We first state Proposition \ref{prop:large-t} that holds true for large $t$.

\begin{proposition} 
    \label{prop:large-t} 
    
    Instate Assumption \ref{assumption:smooth}. 
    Let $\{ \x_t \}_{t\in\mathbb{N}}$ be the sequence generated by Algorithm \ref{alg}. Pick step sizes $\eta_t$ so that there exist $ \eta_-, \eta_+ \in (0, \infty) $ such that $ \eta_- \le \eta_t \le \eta_+ $ for all $t$.Let $ \x_\infty$ be the almost sure limit of $\{ \x_t \}_{t \in \mathbb{N}}$ (Lemma \ref{lem:convergence-pt}). 
    \begin{enumerate}[label=(\textit{\roman*})] 
        \item[] If $\theta \in (0,\frac{1}{2}]$, there exists $T_0 < \infty$ such that 
        \item $ f (\x_t) - f (\x_\infty) \le 1 $ for all $ t \ge T_0 $. 
        \item[] If $\theta \in (\frac{1}{2},1)$, there exist constants $C_0, T_0 < \infty$ such that 
        \item 
        $\E \[ f (\x_t) - f (\x_\infty) \] \in \( 0, \( \frac{1}{ 2\theta B } \)^{\frac{1}{2\theta - 1}} \) $ for all $t \ge T_0$
        \item $C_0 > 1$ and $ \frac{ C_0 }{ 2\theta - 1 } t^{ \frac{2 \theta }{1 - 2\theta} } - B C_0^{2\theta} t^{\frac{2 \theta  }{1 - 2 \theta}} + C \delta_t \le 0$ for all $t \ge T_0$; 
        \item $ \E \[ f (\x_{T_0}) - f (\x_\infty)\] \le C_0 T_0^{ \frac{1 }{1 - 2\theta} } $.  
    \end{enumerate} 
\end{proposition} 


\begin{proof}

    We first prove item $(i)$. 
    When $ \theta \in (0,\frac{1}{2} ] $, $ - z^{2 \theta } \overset{\textcircled{1}}{\le} - \min \{ 1 , z \} $ for all $z \in [0, \infty)$. 

    Let $ \mathcal{X} \subseteq \mathbb{N} $ be a set of times where $ f (\x_s) - f (\x_\infty) > 1$ for $s \in \mathcal{X}$. Suppose, in order to get a contradiction, that $ | \mathcal{X} | = \infty $. 

    \if\submission1{\textcolor{red}{By Proposition \ref{prop:rv-smooth}}}\else{By Proposition \ref{prop:rv-smooth}}\fi, we have 
    \begin{align*} 
        \lim_{t\to \infty } f (\x_{t+1}) 
        \le 
        f (\x_0) - B \if\submission1{\textcolor{red}{\frac{n}{k}}}\else{\frac{n}{k}}\fi \sum_{s \in \mathcal{X} } \nabla f (\x_s)^\top \P_s \nabla f (\x_s) + C \sum_{s=0}^\infty \delta_s . 
    \end{align*} 

    Note that $ \mathcal{X} $ and $ \E_{ \V_{s: s \in \mathcal{X}}} \[ \lim_{t\to \infty } f (\x_{t+1})  \] $ are random variables contained in $ \cup_{s=0}^\infty \mathcal{F}_s $. 
    Thus taking expectation with respect to $ \{ \V_{s}: s \in \mathcal{X} \} $ on both sides of the above inequality gives
    \begin{align*}
        \E_{ \V_{s: s \in \mathcal{X}}} \[ \lim_{t\to \infty } f (\x_{t+1})  \] 
        \le& \;  
        f (\x_0)  + C \sum_{s=0}^\infty \delta_s 
        - B \if\submission1{\textcolor{red}{\frac{n}{k}}}\else{\frac{n}{k}}\fi \sum_{s \in \mathcal{X} } \E_{\V_{s: s \in \mathcal{X}}} \[ \nabla f (\x_s)^\top \P_s \nabla f (\x_s) \] \\ 
        =& \; 
        f (\x_0)  + C \sum_{s=0}^\infty \delta_s 
        -  B \sum_{s \in \mathcal{X} } \| \nabla f (\x_s) \|^2 \\ 
        \le& \; 
        f (\x_0)  + C \sum_{s=0}^\infty \delta_s 
        -  B \sum_{s \in \mathcal{X} } \( f (\x_t ) - f (\x_\infty) \)^{2\theta} \tag{by Proposition \ref{prop:conv-loj}} \\ 
        \le& \; 
        f (\x_0)  + C \sum_{s=0}^\infty \delta_s 
        -  B \sum_{s \in \mathcal{X} } \min \{ 1 , f (\x_t) - f (\x_{\infty}) \} . \tag{by \textcircled{1}} 
    \end{align*}

    \if\submission1{
    \emph{\textcolor{red}{(Compared to the previous version, the last line in the above derivation is removed)}} 
    }\fi 
    
    If there does not exist a $T_0$ such that $ f (\x_t) -  f (\x_{\infty}) \le 1 $ for all $ t \ge T_0 $, then the above implies that $ f (\x_t) $ goes to negative infinity. This is a contradiction to that $ f (\x_t) $ is bounded from below (Assumption \ref{assumption:nonsmooth}). This finishes the proof of item $(i)$.

    Item $(ii)$ follows from Lemma \ref{lem:convergence-pt}. Next we prove items $(iii)$ and $(iv)$.

    Recall $B = -\min \left\{ \( \frac{ L \eta_-^2 n }{ 2k } - \eta_- \), \( \frac{ L \eta_+^2 n }{ 2k } - \eta_+ \) \right\}. $ 
    Consider a $T_0$ so that the following is satisfied for all $t \ge T_0$
    \begin{align} 
        & C \delta_t t^{\frac{2\theta}{2\theta - 1}} \le \frac{1}{2\theta - 1} .  \label{eq:dummy2} 
    \end{align} 
    Next, since $2 \theta > 1$, we can pick $C_0' > 1$ so that 
    \begin{align}
        \frac{ C_0' + 1 }{ 2\theta - 1 } - B (C_0')^{2\theta} \le 0 . \label{eq:dummy3}
    \end{align}
    
    Thus, for $t \ge T_0$, it holds that 
    \begin{align*}  
        \frac{ C_0' }{ 2\theta - 1 } - B (C_0')^{2\theta} + C \delta_t  t^{ \frac{2 \theta }{2\theta - 1} }  
        \le& \; 
        \frac{ C_0' + 1 }{ 2\theta - 1 } - B (C_0')^{2\theta} \le   0 . \tag{by Eqs. \ref{eq:dummy2} and \ref{eq:dummy3}} 
    \end{align*}
   By multiplying both sides of the above inequality by $t^{\frac{2\theta}{1-2\theta}}$, we find a $C_0'$ satisfying $(iii)$. In fact, any constant larger than this $ C_0' $ satisfies this item.

    Finally, we can find $C_0''$ satisfying item $(iv)$, since $ \E \[ f (\x_{T_0}) - f (\x_\infty)\] $ is absolutely bounded for any given $T_0$. Indeed, any constant larger than this $ C_0'' $ satisfies this item. We finish the proof by taking $C_0 = \max \{ C_0', C_0'' \}$. 
\end{proof} 


By Propositions \ref{prop:large-t} and \ref{prop:conv-loj}, (\ref{eq:for-convergence}) implies that, for $t \ge T_0$, 
\begin{align}
    \E_t \[ f (\x_{t+1})  \] - f (\x_\infty) 
    \le& \; 
    f (\x_t) - f (\x_\infty) - B \| \nabla f (\x_t) \|^2 + C \delta_t \nonumber \\
    \le& \; 
    f (\x_t) - f (\x_\infty) - B \( f (\x_t) - f (\x_\infty) \)^{2\theta} + C \delta_t. \label{eq:for-convergence-zero}
\end{align} 

The above inequality (\ref{eq:for-convergence-zero}) is a stochastic relation for the stochastic sequence $ \{ f (\x_t) - f (\x_\infty) \}_{t \in \mathbb{N}} $. In what follows, we will study the convergence behavior of this sequence.



Now we are ready to prove Theorem \ref{thm:rate}.

\begin{proof}[Proof of Theorem \ref{thm:rate}(a)] 
    Let $T_0$ be the constant so that item $(i)$ in Proposition \ref{prop:large-t} is true. 
    By Eq. (\ref{eq:for-convergence}), we have 
    \begin{align}
        \E_{\V_t} \[ f (\x_{t+1}) \] - f (\x_{\infty}) 
        \le& \; 
        f (\x_t) - f (\x_{\infty}) - B \| \nabla f (\x_t) \|^2 + C \delta_t \nonumber \\
        \le& \; 
        f (\x_t) - f (\x_\infty) - B \( f (\x_t) - f (\x_\infty) \)^{ 2\theta } + C \delta_t \label{eq:f-seq} \\ 
        \le& \; 
        \( 1 - B \) \( f (\x_t) - f (\x_\infty) \) + C \delta_t,  \label{eq:cond-exp}
    \end{align} 
    where (\ref{eq:f-seq}) uses Proposition \ref{prop:conv-loj} and the last inequality uses item $(i)$ in Proposition \ref{prop:large-t}. 
    Taking total expectation on both sides of (\ref{eq:cond-exp}) gives 
    \begin{align} 
        & \; \E \[ f (\x_{t+1}) - f (\x_\infty) \] \nonumber \\ 
        \le& \;  
        (1 - B) \E \[ f (\x_t) - f (\x_\infty) \] + C \delta_{t} \nonumber \\
        \vdots&  \nonumber \\
        \le& \;  
        (1-B)^{t- T_0\if\submission1{\textcolor{red}{+1}}\else{+1}\fi} \E \[  \( f (\x_{T_0}) - f (\x_\infty) \) \] + C \if\submission1{\textcolor{red}{\max\{ (1-B)^{t-T_0}, 2^{-(t-T_0)} \} }}\else{ \max\{ (1-B)^{t-T_0}, 2^{-(t-T_0)} \} }\fi  
        \label{eq:linear-rate-exp} . 
    \end{align} 


    Let $Q \in (1, \if\submission1{\textcolor{red}{\min\{ (1-B)^{-1}, 2 \} }}\else{\min\{ (1-B)^{-1}, 2 \}}\fi )$ and let $ Z_t := Q^{t}  \( f (\x_{t} ) - f (\x_\infty) \) $. 
    By (\ref{eq:linear-rate-exp}), we have 
    \begin{align*}
        \E \[ Z_{t+1} \] 
        \le 
        Q^t (1-B)^{t- T_0} Z_{T_0} + C Q^t \if\submission1{\textcolor{red}{ \max \{ (1-B)^{t-T_0} , 2^{-(t-T_0)} \} } }\else{\max \{ (1-B)^{t-T_0} , 2^{-(t-T_0)} \}}\fi . 
    \end{align*} 

    We conclude the proof by noticing 
        $\lim_{t \to \infty} \E \[ Z_{t+1} \] = 0. $
    

\end{proof}

\begin{proof}[Proof of Theorem \ref{thm:rate}(b)] 
    Let $C_0$ and $T_0$ be two constants so that items $(ii)$, $(iii)$ and $(iv)$ in Proposition \ref{prop:large-t} hold true. 

    Taking expectation on both sides of (\ref{eq:for-convergence-zero}) gives, for all $t \ge T_0$,  
    \begin{align*}
        \E \[ f (\x_{t+1}) - f (\x_\infty) \] 
        \le& \;  
        \E \[ f (\x_t) - f (\x_\infty) \] - B \E \[ \( f (\x_t) - f (\x_\infty) \)^{2\theta} \] + C \delta_t \\
        \le& \; 
        \E \[ f (\x_t) - f (\x_\infty) \] - B \E \[  f (\x_t) - f (\x_\infty) \]^{2\theta} + C \delta_t , 
    \end{align*}
     where the last inequality uses Jensen's inequality.

    For simplicity, write $ y_t := \E \[ f (\x_t) - f (\x_\infty) \] $ for all $t$. 
    Next, we use induction to show that 
    \begin{align}
        y_t \le C_0 t^{ \frac{1}{1 - 2 \theta} } , \quad \forall t \ge T_0 . 
        \label{eq:conv-sublin-exp}
    \end{align}

    Suppose that $y_t \le C_0 t^{\frac{1 }{1 - 2 \theta}}$, which is true when $t = T_0$ (item $(iv)$ in Proposition \ref{prop:large-t}). Then for $y_{t+1}$ we have 
    \begin{align} 
        y_{t+1} 
        \le  
        y_t -B y_t^{2\theta} + C \delta_t 
        \le& \; 
        C_0 t^{\frac{1 }{1 - 2 \theta}} -B C_0^{2\theta} t^{\frac{2 \theta }{1 - 2 \theta}} + C \delta_t , \label{eq:critical} 
    \end{align} 
    where second inequality uses item $ (ii) $ in Proposition \ref{prop:large-t} and that the function 
    $z \mapsto z -B z^{2\theta}$
    is strictly increasing when $ z \in \( 0, \( \frac{1}{ 2\theta B } \)^{\frac{1}{2\theta - 1}} \) $.

    
    
    
    
    By applying Taylor's theorem (mean value theorem) to the function $ h (w) = (t + w)^{ \frac{1 }{1 - 2 \theta} } $, we have $ h (1) = h (0) + h' (z) $ for some $z \in [0,1]$. This gives 
    \begin{align} 
        (t+1)^{ \frac{1 }{ 1 - 2 \theta } } 
        = 
        t^{ \frac{1 }{ 1 - 2 \theta } } + \frac{1}{1 - 2\theta} (t + z)^{ \frac{2 \theta }{1 - 2\theta} }
        \ge 
        t^{ \frac{1 }{ 1 - 2 \theta } } + \frac{1}{1 - 2\theta} t^{ \frac{2 \theta }{1 - 2\theta} }, 
        \label{eq:mean-value} 
    \end{align} 
    since $z \in [0,1]$. Thus by (\ref{eq:critical}), we have 
    \begin{align*} 
        y_{t+1} 
        \le& \; 
        C_0 t^{\frac{1  }{1 - 2 \theta}} - B C_0^{2\theta} t^{\frac{2 \theta  }{1 - 2 \theta}} + C \delta_t \\
        \le& \; 
        C_0 (t+1)^{ \frac{1 }{ 1 - 2 \theta } } + C_0 \frac{1}{ 2\theta - 1 } t^{ \frac{2 \theta }{1 - 2\theta} } - B C_0^{2\theta } t^{\frac{2 \theta  }{1 - 2 \theta}} + {C} \delta_t 
        \le 
        C_0 (t+1)^{ \frac{1 }{ 1 - 2 \theta } } , 
    \end{align*} 
    where the second inequality uses (\ref{eq:mean-value}), and the last inequality uses $(iii)$ in Proposition \ref{prop:large-t}. 




\end{proof}

\subsection{Convergence Rate of $ \{ \x_t \}_{t \in \mathbb{N} }$}
\label{sec:conv-rate-x}


In the previous subsection, we have proved convergence rate of $ \{ f (\x_t) \}_{t \in \mathbb{N}} $. 
In this section, we prove convergence rates for $ \{ \| \x_t - \x_\infty \| \}_{t \in \mathbb{N} } $ and $ \{ \sum_{s=t}^\infty \| \x_s - \x_{s+1} \|^2 \}_{t \in \mathbb{N} } $.


\begin{theorem} 
    \label{thm:rate-norm-sq} 
    Instate Assumption \ref{assumption:smooth}. Let $\x_t$ be a sequence generated by the SZGD algorithm, and let $ \x_\infty$ be the almost sure limit of $ \{ \x_t \}_t $. Pick step sizes $\eta_t$ so that there exist $ \eta_-, \eta_+ \in (0, \infty) $ such that $ \eta_- \le \eta_t \le \eta_+ $ for all $t$. Then it holds that 
    \begin{enumerate}[label=(\alph*)] 
        \item if $ \theta \in (0, \frac{1}{2} ] $, $ - 1 \le \( \frac{L n \eta_-^2 }{2k} - \eta_- \) < 0 $ and $ - 1 \le \( \frac{L n \eta_+^2 }{2k} - \eta_+ \) < 0 $, there exists a constant $Q > 1$ such that $ \{ Q^t \sum_{s=t}^\infty \| \x_{s+1} - \x_{s} \|^2 \}_t $ converges to $0$ in expectation. 
        \item if $ \theta \in ( \frac{1}{2}, 1 ) $ and $ \eta_-, \eta_+ \in \( 0, \frac{2k}{Ln} \) $, then it holds that $ \left\{  \sum_{s=t}^\infty \| \x_{s+1} - \x_{s} \|^2 \right\}_t $ converges to $0$ at rate $ O \( t^{ \frac{ 1 }{ 1 - 2 \theta  } } \) $ in expectation. 
    \end{enumerate} 
\end{theorem} 

\begin{proof}
    Note that $ \{ \| \nabla f (\x_t) \| \}_t $ is absolutely bounded due to boundedness of $\{ \x_t \}_t$ and continuity of $ \nabla f (\x) $. 
    By Theorems \ref{thm:variance} and \ref{thm:bias} it holds that 
    \begin{align*} 
        \E_{t} \[ \| \wh{\nabla} f_{\V_t}^{\delta_t} ( \x_t ) \|^2 \] 
        \le& \;  
        \| \E_{t} \[ \wh{\nabla} f_{\V_t}^{\delta_t} ( \x_t ) \] \|^2 + \( \frac{n}{k} - 1 \) \| \nabla f (\x_t) \|^2 \\ 
        &+ \frac{ 4 L \if\submission1{\textcolor{red}{\delta_t}}\else{\delta_t}\fi }{ \sqrt{3} } \( \frac{n^2}{k}  - n \) \| \nabla f (\x_t) \| + \frac{ 4 L^2 n^2 \delta_t^2 }{ 3 k } 
        \le 
        \frac{n}{k} \| \nabla f (\x_t) \|^2 + C \delta_t . 
    \end{align*} 
    
    Thus it holds that 
    \begin{align*} 
        \frac{k}{n} \( \eta_t - \frac{L\eta_t^2 n }{2 k } \) \E_t \[ \| \wh{\nabla} f_{\V_t}^{\delta_t} ( \x_t ) \|^2 \] 
        \le& \; 
         \( \eta_t - \frac{L\eta_t^2 n }{2 k } \) \| \nabla f (\x_t) \|^2 + C \delta_t \\ 
        \le& \; 
        f (\x_t) - \E_t \[ f (\x_{t+1}) \] + C \delta_t , 
    \end{align*} 
    where the second inequality uses (\ref{eq:for-convergence}). 
    
    Since $ \x_{t+1} = \x_t - \eta_t \wh{\nabla} f_{\V_t}^{\delta_t} (\x_t) $, the above inequality implies that, 
    \begin{align} 
        \frac{k}{n} \(  \frac{1}{\eta_t} - \frac{L n }{2 k } \) \E_t \[ \| \x_{t+1} - \x_t \|^2 \] 
        \le
        f (\x_t) - \E_t \[ f (\x_{t+1}) \] + C \delta_t. \label{eq:l-smooth-dis} 
    \end{align}

    Taking total expectation on both sides of the above inequality gives
    \begin{align*} 
        & \; \frac{k}{n} \( \frac{1}{\eta_t} - \frac{L n }{2 k } \) \E \[ \| \x_{s+1} - \x_s \|^2 \] 
        \le 
        \E \[ f (\x_t) - f (\x_\infty) \] - \E \[ f (\x_{t+1}) - f (\x_\infty) \] + C \delta_t . 
    \end{align*} 
    
    Since $ \{ \x_t \}_t $ converges almost surely and $\eta_t \in [\eta_-, \eta_+]$, summing up the above inequality gives 
    \begin{align*} 
        \frac{k}{n} \( \frac{1}{\eta_+} - \frac{L n }{2 k } \) \E \[ \sum_{s=t}^\infty \| \x_{s+1} - \x_s \|^2 \] 
        \le 
        \E \[ f (\x_t) - f (\x_\infty) \] + C \delta_t, \quad \forall t \ge T_0. 
    \end{align*} 
    
    Since we have proved the convergence rate for $\{ \E \[ f (\x_t) - f (\x_\infty) \] \}_t$, we can conclude the proof by the result of Theorem \ref{thm:rate}. 
    

\end{proof} 


Next we provide convergence rate guarantee for $ \left\{ \| \x_t - \x_\infty \| \right\}_t$ in Theorem \ref{thm:rate-norm}. 

\begin{theorem} 
    \label{thm:rate-norm} 
    Instate Assumption \ref{assumption:smooth}. Let $ \x_t$ be a sequence generated by the SZGD algorithm, and let $ \x_\infty$ be the almost sure limit of $ \{ \x_t \}_t $. Pick step sizes $\eta_t$ so that there exist $ \eta_-, \eta_+ \in (0, \infty) $ such that $ \eta_- \le \eta_t \le \eta_+ $ for all $t$. Then it holds that 
    \begin{enumerate}[label=(\alph*)] 
        \item if $ \theta \in (0, \frac{1}{2} ] $, $ - 1 \le \( \frac{L n \eta_-^2 }{2k} - \eta_- \) < 0 $ and $ - 1 \le \( \frac{L n \eta_+^2 }{2k} - \eta_+ \) < 0 $, there exists a constant $Q > 1$ such that $ \{ Q^t \| \x_{t} - \x_{\infty} \| \}_t $ converges to $0$ in expectation. 
        \item if $ \theta \in ( \frac{1}{2}, 1 ) $ and $ \eta_-, \eta_+ \in \( 0, \frac{2k}{Ln} \) $, then it holds that $ \left\{ \| \x_{t} - \x_{\infty} \| \right\}_t $ converges to $0$ at rate $ O \(  t^{ \frac{ 1-\theta }{ 1 - 2 \theta } } \) $ in expectation.  
    \end{enumerate} 
\end{theorem} 


\begin{proof} 
    In \cite{attouch2009convergence}, Attouch and Bolte uses properties of the function $x \mapsto -x^{1-\theta}$ ($x > 0$, $\theta \in (0,1)$) to study the convergence of $ \{ \| \x_t - \x_\infty \| \}_t $. Here we follow a similar path, but in a probabilistic manner.  
    By convexity of the function $x \mapsto -x^{1-\theta}$ ($x > 0$, $\theta \in (0,1)$), it holds that 
    \begin{align} 
        z_2^{1 - \theta} - z_1^{1 - \theta} \ge (1 - \theta ) z_2^{-\theta} ( z_2 - z_1 ) , \qquad \forall z_1, z_2 > 0, \; \; \theta \in (0,1) . \label{eq:from-conv}
    \end{align} 
    By letting $z_1 = f (\x_{t+1} ) - f (\x_\infty)$ and $ z_2 = f (\x_{t}) - f (\x_\infty) $, we have 
    \begin{align*}
        ( f (\x_t) - f (\x_{\infty}) ) ^{1-\theta} - ( f (\x_{t+1}) - f (\x_\infty) ) ^ {1-\theta} 
        \ge& \; 
        (1-\theta) ( f (\x_t) - f (\x_\infty) ) ^{-\theta} \( f (\x_t) - f (\x_{t+1}) \) . 
    \end{align*}
    Taking conditional expectation (conditioning on $\mathcal{F}_t$) on both sides of the above equation gives
    \begin{align}
        & \; ( f (\x_t) - f (\x_\infty) )^{1-\theta} - \E_t \[ ( f (\x_{t+1}) - f (\x_\infty ))^{1-\theta} \] \nonumber \\
        \ge& \; 
        (1-\theta) (f (\x_t) - f (\x_\infty))^{-\theta} \( f (\x_t) - \E_t \[ f (\x_{t+1}) \] \) \nonumber \\
        \ge& \; 
        (1-\theta) ( f (\x_t) - f (\x_\infty) )^{-\theta} 
        \( \frac{k}{n} \(  \frac{1}{\eta_t} - \frac{L n }{2 k } \) \E_t \[ \| \x_{t+1} - \x_t \|^2 \] - C \delta_t \) , \label{eq:for-x-norm-1} 
    \end{align} 
    where the last line uses (\ref{eq:l-smooth-dis}). 

    Also, it holds that, for sufficiently large $t$, 
    \begin{align} 
        0 <&\; 
        ( f (\x_t) - f (\x_\infty) ) ^{\theta} \overset{\1}{\le} \| \nabla f (\x_t) \|  
        \overset{\2}{\le}  
        \left\| \E_t \[ \wh{\nabla} f_{\V_t}^{\delta_t} (\x_t) \] + O \( \frac{ n  \delta_t }{n+1} \bm{1} \) \right\| \nonumber \\
        {\le}& \;  
        \left\| \E_t \[ \wh{\nabla} f_{\V_t}^{\delta_t} (\x_t) \] \right\| + C \delta_t 
        = 
        \frac{1}{\eta_t} \| \E_t \[ \x_{t+1} - \x_t \] \| + C \delta_t  \nonumber \\
        {\le}& \;  
        \frac{1}{\eta_t} \sqrt{ \E_t \[ \| \x_{t+1} - \x_t \|^2 \] } + \sqrt{ C \delta_t } ,  \label{eq:for-x-norm-2} 
    \end{align} 
    where \1 uses Proposition \ref{prop:conv-loj}, and \2 uses Theorem \ref{thm:bias}. \if\submission1\emph{\textcolor{red}{(Compared to the previous version, some typographical square roots before in the above equations are removed. )}}\fi 
    
    Combining the above result with (\ref{eq:for-x-norm-1}) gives
    \begin{align} 
        & \; ( f (\x_t) - f (\x_\infty) )^{1-\theta} - \E_t \[ ( f (\x_{t+1}) - f (\x_\infty) )^{1-\theta} \] \nonumber \\
        \ge& \;  
        \frac{ (1-\theta) 
        \( \frac{k}{n} \( \if\submission1{\textcolor{red}{ \frac{1}{\eta_t} - \frac{L n }{2 k } }}\else{ \frac{1}{\eta_t} - \frac{L n }{2 k } }\fi \) {\eta_t} \E_t \[ \| \x_{t+1} - \x_t \|^2 \] - C \delta_t \) }{ \sqrt{ \E_t \[ \| \x_{t+1} - \x_t \|^2 \] } + \sqrt{ C \delta_t }  } \nonumber \\ 
        =& \; 
        (1-\theta) \frac{k}{n} \( \if\submission1{\textcolor{red}{ 1 - \frac{L n \eta_t }{2 k } }}\else{ 1 - \frac{L n \eta_t }{2 k } }\fi \) \frac{ \E_t \[ \| \x_{t+1} - \x_t \|^2 \] - C \delta_t }{ \sqrt{ \E_t \[ \| \x_{t+1} - \x_t \|^2 \] } + \sqrt{ C \delta_t } } . \label{eq:for-x-norm-3} 
    \end{align} 
    
    
    
    For sufficiently large $t$ such that $C\delta_t < 1$, we have 
    \begin{align*}
        &\; \sqrt{ \E_t \[ \| \x_{t+1} - \x_t \|^2 \] } - \( C \delta_t \)^{1/2}  
        \le 
        \frac{ \E_t \[ \| \x_{t+1} - \x_t \|^2 \] - C \delta_t }{ \sqrt{ \E_t \[ \| \x_{t+1} - \x_t \|^2 \] } + \sqrt{ C \delta_t } } 
        \\
        \le& \; 
        \frac{ n }{ (1-\theta) k \( 1 - \frac{L n \eta_t }{ 2 k } \) } \( ( f (\x_t) - f (\x_\infty) )^{1-\theta} - \E_t \[ ( f (\x_{t+1}) - f (\x_\infty) )^{1-\theta} \] \) 
    \end{align*} 
    where the last inequality uses (\ref{eq:for-x-norm-3}). 
    
    Combining Jensen's inequality and the above inequality gives
    \begin{align*} 
        &\; \E_t \[ \| \x_{t+1} - \x_t \| \] - \( C \delta_t \)^{1/2}  
        \le 
        \sqrt{ \E_t \[ \| \x_{t+1} - \x_t \|^2 \] } - \( C \delta_t \)^{1/2}  \\ 
        \le& \; 
        \frac{ n }{ (1-\theta) k \( 1 - \frac{L n \eta_t }{ 2 k } \) } \( ( f (\x_t) - f (\x_\infty) )^{1-\theta} - \E_t \[ ( f (\x_{t+1}) - f (\x_\infty) )^{1-\theta} \] \) . 
    \end{align*} 
    
    Taking total expectation on both sides of the above inequality and summing over times gives
    \begin{align*}
        \sum_{s = t}^\infty \E \[ \| \x_{s+1} - \x_s \| \] 
        \le& \;  
        \frac{ n }{ (1-\theta) k \( 1 - \frac{L n \eta_t }{ 2 k } \) } \E \[ ( f (\x_t) - f (\x_\infty) )^{1-\theta} \] + C \sum_{s=t}^\infty \delta_s^{1/2} \\
        =& \; 
        \frac{ n }{ (1-\theta) k \( 1 - \frac{L n \eta_t }{ 2 k } \) } \E \[ ( f (\x_t) - f (\x_\infty) )^{1-\theta} \] + C \delta_t^{1/2} 
    \end{align*} 
    Since $ \| \x_t - \x_\infty \| \le  \sum_{s = t}^\infty \| \x_{s+1} - \x_s \| $, the above implies that 
    \begin{align*} 
        \E \[ \| \x_t - \x_\infty \| \] 
        \le& \;  
        \frac{ n}{ (1-\theta) k \( 1 - \frac{L n \eta_t }{ 2 k } \) } \E \[ ( f (\x_t) - f (\x_\infty) )^{1-\theta} \] + C \delta_t^{1/2} \\
        \le& \; 
        \frac{ n }{ (1-\theta) k \( 1 - \frac{L n \eta_t }{ 2 k } \) } \E \[ f (\x_t) - f (\x_\infty) \]^{1-\theta}  + C \delta_t^{1/2} ,
    \end{align*}
    where the last inequality uses Jensen’s inequality. We can conclude the proof by the result of Theorem \ref{thm:rate}. 

    \if\submission1
    \emph{\textcolor{red}{(in the above equations, some typographical appearance of ${\eta_t}$ is removed.)}} 
    \fi 

\end{proof}

\subsection{Implications on Gradient Descent} 


The convergence rate for SZGD implies sure convergence rate of the gradient descent algorithm. 
In this section, we will display convergence rate results for the gradient descent algorithm on \Loj functions. 
Note that the classic work \cite{attouch2009convergence} provides convergence rate for the proximal algorithm, not the gradient descent algorithm. Also \cite{karimi2016linear} provides analysis for the gradient descent on the Polyak--\Loj functions, not the \Loj functions. 
Thus this convergence rate of $\{ f (\x_t) \}_t$ governed by gradient descent for smooth \Loj functions is one of our contributions, although it may not be as important as the results in previous sections. 


Recall the gradient descent algorithm iterates as 
\begin{align} 
    \x_{t+1} = \x_t - \eta_t \nabla f(\x_t). \tag{Gradient Descent (GD)}
\end{align} 
Compare to GD, the SZGD algorithm does not require one to have access to first-order information of the objective. In this sense, SZGD algorithm makes weaker assumptions about the environment. 

\begin{corollary} 
    \label{cor:non-stoc-smooth-linear} 
    Instate Assumption \ref{assumption:smooth}. Let $\x_t$ be a sequence generated by the gradient descent algorithm, and let $\x_\infty$ be the limit of $ \{ \x_t \}_t $. Suppose there exists $\eta_-, \eta_+ \in (0,\infty)$ such that $ \eta_- \le \eta_t \le \eta_+ $ for all $t$. 
    Then if $ \theta \in (0, \frac{1}{2} ] $, $ - 1 \le \( \frac{L \eta_-^2 }{2} - \eta_- \) < 0 $ and $ - 1 \le \( \frac{L \eta_+^2 }{2} - \eta_+ \) < 0 $, there exists a constant $Q > 1$ such that $ \{ Q^t \| \x_{t} - \x_{\infty} \| \}_t $ converges to $0$ and $ \{ Q^t \( f (\x_{t}) - f ( \x_{\infty} ) \) \}_t $ converges to $0$. 
\end{corollary} 

Corollary \ref{cor:non-stoc-smooth-linear} follows immediately from Theorems \ref{thm:rate} and \ref{thm:rate-norm}. 
If $\theta \in (0,\frac{1}{2}]$, Corollary \ref{cor:non-stoc-smooth-linear} provides linear convergence rate of $\{ \| \x_t - \x_\infty \| \}_t$ and $\{ f ( \x_t ) - f ( \x_\infty ) \}_t$ with $x_t$ governed by the GD algorithm. When $\theta \in (0,\frac{1}{2}]$, linear convergence rate for $\{ \sum_{s=t}^\infty \| \x_{s+1} - \x_s \| \}_t$ can be similarly obtained. 

For $ \theta \in (\frac{1}{2} , 1) $, the convergence result is summarized in Theorem \ref{thm:conv-nonstoc}. 
To prove Theorem \ref{thm:conv-nonstoc}, we first need the following proposition. 


\begin{proposition}
    \label{prop:large-t-nonstoc} 
    Instate Assumption \ref{assumption:smooth}. Let $\{ \x_t \}_{t\in\mathbb{N}}$ be the sequence governed by the gradient descent algorithm. Pick step sizes $\eta_t$ so that there exist $ \eta_-, \eta_+ \in (0, \infty) $ such that $ \eta_- \le \eta_t \le \eta_+ $ for all $t$. Let $\{ \x_t \}_{t \in \mathbb{N}}$ be bounded and let $ \x_\infty$ be the limit of $\{ \x_t \}_{t \in \mathbb{N}}$. 
    Then there exist constants $C_0$ and $T_0$ such that the following holds: 
    \begin{enumerate}[label=(\textit{\roman*})] 
        \item $f (\x_t) - f (\x_\infty) \in \( 0, \( \frac{1}{ 2\theta B }  \)^{\frac{1}{2\theta - 1}} \)$ for all $t \ge T_0$.  
        \item $ - C_0^{2\theta} B + \frac{ C_0 }{2\theta - 1}  \le 0 $, where $B$ is defined in (\ref{eq:def-B}). 
        \item $ f (\x_{T_0}) - f (\x_\infty) \le C_0 T_0^{ \frac{1}{1 - 2\theta} } $. 
    \end{enumerate} 
\end{proposition} 

\begin{proof} 

    Since $f$ is $L$-smooth, it holds that 
    \begin{align*} 
        f (\x_{t+1}) 
        \le& \;  
        f (\x_t) + \nabla f (\x_t)^\top (\x_{t+1} - \x_t) + \frac{L}{2} \| \x_t - \x_{t+1} \|^2 \\ 
        \le& \; 
        f (\x_t) - B \| \nabla f (\x_t) \|^2 ,  
    \end{align*} 
    where $B$ is defined in (\ref{eq:def-B}). 

    Therefore, $\lim_{t \to \infty} \| \nabla f (\x_t) \| = 0$. Otherwise, the above inequality leads to a contradiction to that $ \inf_{\x} f (\x) > -\infty $. By Lemma \ref{lem:convergence-pt}, we know that $ \{ \x_t \}_t $ converges to a critical point of $f$. Thus by continuity of $f$, item $(i)$ can be satisfied since $ \{ f (\x_t) \}_t$ converges. 
    Item $(ii)$ can be satisfied by picking some $ C_0 > 1 $. 
    Since $ \{ f (\x_t) \}_t $ is bounded, we can pick $ C_0 $ large enough so that item $ (iii) $ is satisfied. 
    
\end{proof}


With the above proposition, we are ready to prove Theorem \ref{thm:conv-nonstoc}. 

\begin{theorem}
    \label{thm:conv-nonstoc} 
    Instate Assumption \ref{assumption:smooth}. Let $\x_t$ be a sequence generated by the gradient descent algorithm, and let $\x_\infty$ be the limit of $ \{ \x_t \}_t $. 
    Pick step sizes $\eta_t$ so that there exist $ \eta_-, \eta_+ \in (0, \infty) $ such that $ \eta_- \le \eta_t \le \eta_+ $ for all $t$.
    If $ \theta \in (\frac{1}{2} ,1) $, then $ f (\x_t) - f (\x_\infty) \le O \( t^{\frac{1}{1-2\theta}} \) $. 
\end{theorem} 

\begin{proof} 
    
    Let $T_0$ and $C_0$ be the constants so that Proposition \ref{prop:large-t-nonstoc} holds true. 
    
    By $L$-smoothness by $f$, it holds that 
    \begin{align*} 
        f (\x_{t+1}) 
        \le& \;  
        f (\x_t) + \nabla f (\x_t)^\top \( \x_{t+1} - \x_t \) + \frac{L}{2} \| \x_t - \x_{t+1} \|^2 \\
        =& \; 
        f (\x_t) + \( \frac{L\eta_t^2}{2} - \eta_t \) \| \nabla f (\x_t) \|^2. 
    \end{align*}
    
    Since $x_t$ is close to a critical point for all $t \ge T_0$ (item $ (i) $ in Proposition \ref{prop:large-t-nonstoc}), combining the above inequality and the \Loj inequality gives
    \begin{align} 
        f (\x_{t+1}) - f (\x_\infty)
        \le& \; 
        f (\x_t) - f (\x_\infty) + \( \frac{L\eta_t^2}{2} - \eta_t \) ( f (\x_t) - f (\x_\infty) )^{2\theta} 
        \label{eq:for-ind-nonstoc} \\ 
        \le& \; 
        f (\x_t) - f (\x_\infty) - B ( f (\x_t) - f (\x_\infty) )^{2\theta} \nonumber, 
    \end{align} 
    where $B$ is defined in (\ref{eq:def-B}). 

    
    For a given $t \ge T_0$, consider the function $ h (x) = (t+x)^{\frac{1}{1-2\theta}} $. Applying Taylor's theorem to $h$ gives
    $
        (t+1)^{\frac{1}{1-2\theta}} 
        = 
        t^{\frac{1}{1-2\theta}} + \frac{1}{1-2\theta} (t+z)^\frac{2\theta}{1-2\theta}  
    $ 
    for some $z \in [0,1]$, which implies
    \begin{align}
        (t+1)^{\frac{1}{1-2\theta}} 
        \ge
        t^{\frac{1}{1-2\theta}} + \frac{1}{1-2\theta} t^\frac{2\theta}{1-2\theta} .  \label{eq:taylor-nonstoc}
    \end{align}

    We will use induction to finish the proof. Item $(iii)$ in Proposition \ref{prop:large-t-nonstoc} states $ f (\x_{T_0}) - f (\x_\infty) \le C_0 {T_0}^{\frac{1}{1 - 2\theta}} $. Inductively, if $f (\x_{t}) - f (\x_\infty) \le C_0 {t}^{\frac{1}{1 - 2\theta}} $ ($t \ge T_0$), then (\ref{eq:for-ind-nonstoc}) implies that 
    \begin{align*} 
        &\; f (\x_{t+1}) - f (\x_\infty) 
        \le
        C_0 t^{\frac{1}{1-2\theta}} - B C_0^{2\theta} t^{ \frac{2\theta}{1-2\theta} } \\
        \le& \; 
        C_0 (t+1)^{\frac{1}{1-2\theta}} - B C_0^{2\theta} t^{ \frac{2\theta}{1-2\theta} } + \frac{ C_0 }{2\theta - 1} t^{\frac{2\theta}{1-2\theta}} 
        \le 
        C_0 (t+1)^{\frac{1}{1-2\theta}} , 
    \end{align*} 
    where the first line uses (\ref{eq:for-ind-nonstoc}), the induction hypothesis, item $(i)$ in Proposition \ref{prop:large-t-nonstoc}, and that the function $x \mapsto x - B x^{2\theta} $ is strictly increasing on $ \( 0, \( \frac{1}{ 2\theta B }  \)^{\frac{1}{2\theta - 1}} \) $, the second line uses (\ref{eq:taylor-nonstoc}), and the last line uses item $(ii)$ in Proposition \ref{prop:large-t-nonstoc}. 
    
\end{proof}

\section{The Proximal Algorithm for Nonsmooth \Loj Functions} 

\label{sec:nonsmooth} 


Let $f$ be a function that is continuous but possibly nonsmooth. In such cases, we use the (subgradient) proximal algorithm to solve this optimization problem. This section serves to provide a convergence rate analysis for the proximal algorithm for nonsmooth \Loj functions with exponent $\theta \in (0,\frac{1}{2})$. Previously, \cite{attouch2009convergence} showed that, for convex nonsmooth \Loj functions, when the \Loj exponent $\theta \in (\frac{1}{2},1)$, the sequence $\{ \x_t \}$ governed by the proximal algorithm (\ref{eq:proximal}) converges at rate $ \| \x_t - \x_\infty \| \le O \( t^{ \frac{1 - \theta}{1- 2 \theta}} \) $. In this section, we show that the proximal algorithm satisfies $ f (\x_t) - f (\x_\infty) \le O \( t^{ \frac{1 }{1- 2 \theta}} \) $. When $\theta \in (\frac{1}{2}, 1)$ the function value $ \{ f (\x_t) \}_t$ tends to converge at a faster rate than the point sequence $\{ \| \x_t - \x_\infty \| \}_t$. This phenomenon for the proximal algorithm again suggests that the convergence rate of $ \{ f (\x_t) \}_t $ may be more important and informative than the convergence rate of $ \{ \| \x_t - \x_\infty \| \}_t $, since the trajectory of $ \{ \x_t \}_t $ may be inevitably spiral \cite{https://doi.org/10.1112/blms.12586}. 

\subsection{Preliminaries for Nonsmooth Analysis and Proximal Algorithm} 

Before proceeding, we first review some preliminaries for nonsmooth analysis and the proximal algorithm. We begin by the concept of subdifferential and subgradient in nonsmooth analysis.

\begin{definition}[\cite{rockafellar2009variational}]
    \label{def:subdifferential}
    Consider a proper lower semicontinuous function $f : \R^n \to \R \cup \{ + \infty \} $. The effective domain (or simply domain) of $f$ (written $\dom f$) is $\dom f := \{ \x \in \R^n : -\infty < f (\x) < +\infty \}$. For each $\x \in \dom f$, the Fr\'echet subdifferential of $f$ at $\x$, written $\hat{\partial} f (\x)$, is the set of vectors $\mathbf{g}^* \in \R^n$ such that 
    \begin{align*} 
        \liminf_{\substack{\y \neq \x \\ \y \to \x}} \frac{ f (\y) - f (\x) - \< \mathbf{g}^* , \y - \x \> }{\| \x - \y\|} \ge 0. 
    \end{align*} 
    If $\x \notin \dom f$, by convention $\hat{\partial} f (\x) =  \emptyset $. 
    
    The limiting subdifferential of $f$ at $\x$, written $\partial f (\x)$, is 
    \begin{align*} 
        \partial f (\x) 
        := 
        \{ \mathbf{g} \in \R^n : \exists \x_n \to \x, f (\x_n) \to f (\x), \mathbf{g}_n^* \in \hat{\partial} f (\x_n) \to \mathbf{g} \} . 
    \end{align*} 
    An element $\mathbf{g} \in \partial f (\x)$ is call a subgradient of $f$ at $\x$. 
    If $f$ is convex, then 
    \begin{align*}
        \partial f(\x) = \bigcap_{ \mathbf{z} \in \dom f }  \{ \mathbf{g} | f(\mathbf{z}) \ge f(\x) + \mathbf{g}^\top  ( \mathbf{z} -  \x ) \}.
    \end{align*}
\end{definition}


Next we review the elements of the proximal algorithm. The proximal algorithm is described by the following inclusion recursion:  
\begin{align}
    \x_{t+1} \in \arg \min_{\mathbf{z}} \left\{ f (\mathbf{z}) + \frac{1}{ 2 \eta_t} \| \mathbf{z} - \x_t \|^2 \right\} , \label{eq:proximal} 
\end{align} 
with a given $\x_0 \in \R^n$. 


In each iteration, the proximal algorithm solves an optimization problem whose solution set is compact and nonempty \cite{attouch2009convergence}. By the optimality condition (Theorem 10.1, \cite{rockafellar2009variational}), we know 
    $0 \in \partial \( f ( \x_{t+1}) + \frac{1}{ 2 \eta_t} \| \x_{t+1} - \x_t \|^2 \) .$
By the subadditivity property of subdifferential (e.g., Exercise 10.10, 
\cite{rockafellar2009variational}), we have 
\begin{align*}
    0 \in \partial \( f ( \x_{t+1}) + \frac{1}{ 2 \eta_t} \| \x_{t+1} - \x_t \|^2 \) \subseteq \{ \frac{1}{\eta} \( \x_{t+1} - \x_t \) \} + \partial f ( \x_{t+1}) , 
\end{align*}
where $+$ is the Minkowski sum when two summands are sets. Thus it holds that 
\begin{align} 
    \x_{t+1} = \x_t - \eta_t \mathbf{g}_{t+1}, \label{eq:proximal-update} 
\end{align} 
for some $\mathbf{g}_{t+1} \in \partial f ( \x_{t+1})$. 

We also recall the following two theorems.

\begin{theorem}[\cite{attouch2009convergence}] 
    \label{thm:ab} 
    Let $f$ satisfy Assumption \ref{assumption:nonsmooth} and let $\{ \x_t \}_{t \in \mathbb{N}}$ be generated by the proximal algorithm. If $\{ \x_t \}_{t \in \mathbb{N}}$ is bounded then it converges to a critical point of $ f $. 
\end{theorem} 

\begin{theorem}[\cite{attouch2009convergence}] 
    \label{thm:ab-rate} 
    Let $f$ satisfy Assumption \ref{assumption:nonsmooth} and let $\{ \x_t \}_{t \in \mathbb{N}}$ be generated by the proximal algorithm. Let $f$ satisfy the \Loj inequality with \Loj exponent $\theta \in (\frac{1}{2}, 1)$. If $\{ \x_t \}_{t \in \mathbb{N}}$ is bounded, 
    then it holds that 
    \begin{align*}
        \| \x_t - \x_\infty \| \le O  ( t^{\frac{1 - \theta}{1 - 2\theta}} ) \quad \text{and} \quad \| \x_t - \x_{t+1} \| \le O  ( t^{\frac{1 - \theta}{1 - 2\theta}} ) , 
    \end{align*}
    where $ x_\infty $ is the limit of $ \{ \x_t \}_t $. 
\end{theorem} 

Same as \cite{attouch2009convergence}, we will make the following assumptions on $f$. All items in Assumption \ref{assumption:nonsmooth} are assumed in \cite{attouch2009convergence}. 

\begin{assumption}[\cite{attouch2009convergence}] 
    \label{assumption:nonsmooth}
    The function $f$ satisfies
    \begin{enumerate} 
        \item $f$ is continuous on $\dom f$; \if\submission1 \textcolor{red}{($\mathcal{H}_2$ in \cite{attouch2009convergence})} \fi 
        \item For any $ \x_* \in \R^n $ with $ \partial f \ni \0 $, it holds that: there exist $ \kappa, \mu > 0$, and $\theta \in [0,1)$, such that 
        \begin{align} 
            | f (\x) - f (\x_*) |^\theta 
            \le 
            \kappa \| \mathbf{g} \|, \quad \forall x \in B (\x_*, \mu) , \forall g \in \partial f (\x), \label{eq:loj-nonsmooth}
        \end{align} 
        where $B (\x_*, \mu)$ is the ball of radius $\mu$ centered at $\x_*$. Without loss of generality, we let $\kappa = 1$ to avoid clutter. \if\submission1
        \textcolor{red}{(The \Loj inequality. Appeared in Eq. 5 in \cite{attouch2009convergence}.)} \fi 
        \item $\inf_{\x \in \R^n} f (\x) > -\infty$; 
        \if\submission1
        \textcolor{red}{($\mathcal{H}_1$ in \cite{attouch2009convergence})} \fi 
        \item Let $ \{ \x_t \}_t $ be the sequence generated by the GD algorithm. We assume $ \{ \x_t \}_t $ is bounded. 
        \if\submission1
        \textcolor{red}{(Stated as a condition in Proposition 1, below item (iii), in \cite{attouch2009convergence}.)} \fi 
        \item There exists $ \eta_-, \eta_+ \in (0,\infty) $ such that $  \eta_- \le \eta_t \le \eta_+  $ for all $t$. 
        \if\submission1
        \textcolor{red}{(Below Eq. 3 in \cite{attouch2009convergence}.)} \fi 
    \end{enumerate} 
\end{assumption} 

\if\submission1 
\textcolor{red}{(The red remarks are added for review purpose only.)}
\fi 

In the above assumption, (\ref{eq:loj-nonsmooth}) is the \Loj inequality. Compared to the one in Definition \ref{def:loj}, gradient is replaced by subgradient. 


\subsection{Convergence of the Proximal Algorithm}

Similar to the stochastic case, we focus on the convergence analysis for $\{ f (\x_t) \}_t$. We start with the following proposition.

\begin{proposition} 
    \label{prop:conv-like} 
    Let $ \{ \x_t \}_{t \in \mathbb{N} } $ be the bounded sequence generated by the proximal algorithm. Let $ \{ \mathbf{g}_t \}_{t \in \mathbb{N} } $ be the sequence of subgradients that governs the iteration (\ref{eq:proximal-update}). Then there exists a sequence $ \{ w_t \}_{t} \subseteq [0, \infty ) $ such that 
    \begin{itemize}
        \item $\lim\limits_{t \to \infty} \frac{w_t}{ \| \x_t - \x_{t+1} \| } = 0 $; and $ f (\x_t) \ge f (\x_{t+1}) + \mathbf{g}_{t+1}^\top \( \x_t - \x_{t+1} \) - w_t, \; \forall t \in \mathbb{N} .  $
        \item If $ f$ is convex, then $ f (\x_t) \ge f (\x_{t+1}) + \mathbf{g}_{t+1}^\top \( \x_t - \x_{t+1} \) $. 
    \end{itemize}
\end{proposition}


\begin{proof}
    
    
    
    By Definition \ref{def:subdifferential}, we have 
    \begin{align*}
        f (\x_t) - f (\x_{t+1}) - \< \mathbf{g}_{t+1} , \x_t - \x_{t+1} \> 
        \ge - o \( \| \x_t - \x_{t+1} \| \) , \quad \forall t, 
    \end{align*} 
    which concludes the proof for the first item. Similarly by Definition \ref{def:subdifferential}, we can prove the second item. 
    
\end{proof}



Next, we state below some numerical properties when $t$ is large. 

\begin{proposition} 
    \label{prop:large-t-nonsmooth} 
    For any $\theta \in (\frac{1}{2} , 1)$, there exists constants $ T_0 $ and  $C_0$ such that 
    \begin{enumerate}[label=\textit{(\roman*)}] 
        \item $ \frac{ 2 }{ \eta_+ } \le \frac{ 2 }{ \eta_- } \le (2\theta - 1) C_0^{2\theta - 1} $; 
        \item For all $ t \ge T_0 $, it holds that 
            $\| \x_t - \x_\infty \| \le \mu, $
        where $\mu$ is defined as in Definition \ref{assumption:nonsmooth}. 
        \item For all $t \ge T_0$, it holds that
        \begin{align*} 
            \frac{ C_0 }{ ( 2\theta - 1 ) } t^{ \frac{ 2 \theta  }{ 1 - 2\theta } } \le \eta_- C_0^{2\theta} (t+1)^{ \frac{ 2 \theta }{ 1 - 2\theta } } \le \eta_t C_0^{2\theta} (t+1)^{ \frac{ 2 \theta }{ 1 - 2\theta } } . 
        \end{align*} 
        \item $ f (\x_{T_0}) - f (\x_\infty) \le C_0 T_0^{ \frac{ 1 }{1 - 2 \theta } } $. 
    \end{enumerate} 
\end{proposition} 

\begin{proof}


    
    Clearly we can find a constant $C_0'$ such that $ 2 \eta_+ \le (2\theta - 1) C^{2\theta - 1} $ for all $C \ge C_0'$, and this $C_0'$ does not depend on $T_0$. Thus item $(i)$ can be easily satisfied. 
    By Theorem \ref{thm:ab}, we can find $T_0$ so that item $(ii)$ is true. 

    By item $(i)$, for any $C \ge C_0'$, it holds that  
        $ 
        \lim\limits_{t \to \infty} \frac{ \frac{ C }{ 2\theta - 1  } t^{ \frac{ 2 \theta  }{ 1 - 2\theta } } }{ \eta_- C^{2\theta} (t+1)^{ \frac{ 2 \theta }{ 1 - 2\theta } } } 
        = 
        \frac{ \frac{ C }{ \eta_- (2\theta - 1 ) } }{ C^{2\theta} } \le \frac{1}{2} . 
        $
    Thus we can find $T_0$ and $C_0'$ that satisfies item $ (iii) $. 
    For item $(iv)$, given a $T_0$, we can find $C_0''$ such that $ f (x_{T_0}) \le C_0'' T_0^{ \frac{1 }{1 - 2 \theta } } $, since the sequence $ \{ f (x_t) \}_{t \in \mathbb{N}} $ is absolutely bounded (Theorem \ref{thm:ab}). 
    Letting $C_0 = \max \{ C_0', C_0'' \}$ concludes the proof. 
\end{proof}

\begin{theorem} 
    \label{thm:rate-nonsmooth}
    Instate Assumption \ref{assumption:nonsmooth}. 
    Let $\{ \x_t \}_{t \in \mathbb{N}}$ generated by the proximal algorithm. If $\{ \x_t \}_{t \in \mathbb{N}}$ is bounded then it converges to a critical point of $ f $. Let $ \x_\infty$ be the limit of $\{ \x_t \}_{t \in \mathbb{N}}$. If $\theta \in (\frac{1}{2},1)$ and $f $ is convex, then it holds that 
    $    f (\x_t) - f ( \x_\infty ) \le O \( t^{ \frac{1}{1-2\theta} } \) . $
\end{theorem}


\begin{proof} 
    
    Since $f$ is convex, we have 
    \begin{align*} 
        f (\x_{t}) 
        \ge& \;  
        f (\x_{t+1}) + \mathbf{g}_{t+1}^\top \( \x_t - \x_{t+1} \) 
        =
        f (\x_{t+1}) + {\eta_t} \| \mathbf{g}_{t+1} \|^2 . 
    \end{align*} 
    
    
    Since $ \{ \x_t \} $ converges (Theorem \ref{thm:ab}) and the \Loj inequality holds with exponent $\theta$, there exists $T_0$, such that for all $t \ge T_0$,  
    \begin{align*} 
        f (\x_t) 
        \ge 
        f (\x_{t+1}) + {\eta_t} \| \mathbf{g}_{t+1} \|^2 
        \ge 
        f (\x_{t+1}) + {\eta_t} \( f (\x_{t+1}) - f (\x_\infty) \)^{2\theta} ,  
    \end{align*} 
    which gives
    \begin{align} 
        f ( \x_t) - f ( \x_\infty) 
        \ge 
        f ( \x_{t+1}) - f ( \x_\infty) + {\eta_t} \( f ( \x_{t+1}) - f ( \x_\infty) \)^{2\theta}. \label{eq:recur-nonsmooth} 
    \end{align} 
    

    
    
    For any positive integer $t$, define
        $h (s) = (t + s )^{ \frac{ 1 }{1 - 2\theta} } .$ 
    By mean value theorem, one has $h (1) = h (0) + h'(z)$ for some $z \in [0,1]$. Thus we have 
    \begin{align} 
        (t+1)^{ \frac{ 1 }{1 - 2\theta} } 
        = 
        t^{ \frac{ 1 }{1 - 2\theta} } + \frac{ 1 }{1 - 2\theta} ( t + z )^{ \frac{ 2\theta }{1 - 2\theta} } 
        \ge 
        t^{ \frac{ 1 }{1 - 2\theta} } + \frac{ 1 }{1 - 2\theta} t^{ \frac{ 2\theta }{1 - 2\theta} } . \label{eq:for-ind-nonsmooth} 
    \end{align}

    Next we use induction to prove the convergence rate. By item $ (iv) $ in Proposition \ref{prop:large-t-nonsmooth}, we can find $C_0$ and $T_0$ such that $ f (\x_{T_0}) - f (\x_\infty) \le C_0 T_0^{ \frac{1}{1 - 2 \theta} } $. Inductively, if $f (\x_t) - f (\x_\infty) \le C_0 t^{\frac{1}{1-2\theta}} $ ($t \ge T_0$), then by (\ref{eq:recur-nonsmooth}) it holds that 
    \begin{align} 
        &\; f (\x_{t+1}) - f (\x_\infty) + \eta_t ( f (\x_{t+1}) - f (\x_\infty) )^{2\theta} 
        \le  
        f (\x_t) - f (\x_\infty)  \nonumber \\
        \le& \;  
        C_0 t^{\frac{1}{1-2\theta}} 
        \le  
        C_0 (t+1)^{\frac{1}{1-2\theta}} + C_0 \frac{1 }{2\theta - 1} t^{\frac{2\theta}{1-2\theta}} \nonumber \\ 
        \le& \;  
        C_0 (t+1)^{\frac{1}{1-2\theta}} + \eta_t C_0^{2\theta} (t+1)^{\frac{2\theta}{1-2\theta}}, \label{eq:for-monotone} 
    \end{align} 
    where the second last inequality uses (\ref{eq:for-ind-nonsmooth}) and the last inequality uses $(iii)$ in Proposition \ref{prop:large-t-nonsmooth}. 
    
    Since the function $x \mapsto x + \eta_t x^{2\theta}$ is monotonic and strictly increasing on $(0, \infty)$, the above inequality (\ref{eq:for-monotone}) implies 
    $
        f ( \x_{t+1}) - f ( \x_\infty) \le C_0 (t+1)^{\frac{1}{1-2\theta}} , 
    $ 
    which concludes the proof. 
\end{proof}

\begin{remark}
    Recently, \cite{https://doi.org/10.1112/blms.12586} constructed a convex \Loj function on the plane that fails the Thom's gradient conjecture. The result in \cite{https://doi.org/10.1112/blms.12586} is somewhat pessimistic since it shows that the gradient flow of \Loj functions can be inevitably spiral even if the function is convex. By Theorems \ref{thm:rate} and \ref{thm:rate-nonsmooth}, we show that the function value $\{ f (\x_t) - f (\x_\infty) \}_{t \in \mathbb{N}}$ can converge at a rate much faster than $ \| \x_t - \x_\infty \| $ (when $\theta \in (\frac{1}{2}, 1)$). This suggests that, for \Loj functions, even though the gradient flow is spiral, the function values can converge at a reasonable rate. 
    
\end{remark}

\section{Empirical Studies} 
\label{sec:exp} 

In this section, we empirically study the performance of the SZGD algorithm. All experiments are carried out on the following test functions $ F_1 $ and $F_2$ defined over $\R^{30}$ 
\begin{align*} 
    F_1 (\x) = \( \x^\top \mathbf{Q} \x \)^{3/4} \quad \text{and} \quad F_2 (\x) = \( \x^\top \mathbf{Q} \x \)^{1/4} , 
\end{align*} 
where $Q \in \R^{30 \times 30}$ is a PSD matrix with eigenvalues following exponential distribution with $p.d.f.$ $f_{exp}(x) = \frac{1}{5} \exp \( - \frac{x}{5} \) \Ind_{ \[ x \ge 0 \] }$ and eigenvectors independently sampled from the unit sphere. The results are summarized in Figures \ref{fig:F} and \ref{fig:sample}. To avoid numerical instability, we set $\delta_t = \max \( 0.1 \times 2^{-t} , 0.00001 \)$ in all numerical experiments. 

\begin{remark}
    In this experiments, we compare the results of SZGD with GD. Strictly speaking, this is not a fair comparison since SZGD uses only zeroth-order information while GD can use first-order information. 
\end{remark}



Key observations from Figure \ref{fig:F} are: 
\begin{itemize}[align=left,leftmargin=*]  
    \item In general, $ \{ f (\x_t) \}_t $ converges faster than $\{ \| \x_t - \x_\infty \| \}_t$, which agrees with our theoretical results. 
    \item Figure \ref{fig:f1-1} and \ref{fig:f1-2} show that, on test function $F_1$, with fixed step size $\eta$, SZGD with larger values of $k$ in general converges faster given a fixed number of iterations. However, the gap between different choices of $k $ is not significant. 
\end{itemize}

Figure \ref{fig:F} also provides other insights. To gain some intuition, we consider the landscape of the test function $ F_2 $. While fully visualizing $F_2$ is impossible, we can consider a 1-d version of $ F_2 $: $f (x) = \sqrt{|x|}$, whose graph is shown below in Figure \ref{fig:toy}. Note that $ f (x) = \sqrt{|x|} $ is indeed a 1-d version of $F_2$, since $ f (x) = \sqrt{|x|} = (x^2)^{1/4} $. This nonconvex function has the following properties, all highly aligned with empirical observations. 
\begin{itemize}[align=left,leftmargin=*] 
    \item When $ \x $ is far from the origin, the function's surface is relatively flat. In this case, randomness in the gradient, which may increase the magnitude of gradient, can speedup the convergence. 
    In Figure \ref{fig:f2-1}, we observe that, \emph{in terms of $\| \x_t - \x_\infty\|$}, SZGD algorithms converge faster than GD at the beginning, which agrees with the landscape of $F_2$. More importantly, \emph{in terms of $ f (\x_t) - f ( \x_\infty )$}, SZGD converges faster than GD, as shown in Figure \ref{fig:f2-2}. 
    \item At zero, the function is not differentiable, and thus not $L$-smooth. This means that the gradient estimator may not be accurate near zero. In Figure \ref{fig:f2-1}, we observe that GD converges to a more optimal value \emph{in terms of $\| \x - \x_\infty \|$} at the end, which agrees with the non-differentiability of $F_2$ at zero. 
    \item Near zero, the gradient of the function is not continuous. Therefore, the trajectory of the GD algorithm oscillates near zero. For the 1-d example in Figure \ref{fig:toy}, the trajectory of GD will jump back and forth around zero, but hits zero with little chance. In Figure \ref{fig:f2-1}, we observe that the $\| \x \|$ value of DG is fuzzy as it goes to zero, which agrees with the shape of $F_2$ near zero. 
\end{itemize} 

We also empirically study the convergence rates versus the number of function evaluations. Note that one iteration may require more than one function evaluations: $k$ random orthogonal directions are sampled and $2k$ function evaluations are needed to obtain the gradient estimator defined in (\ref{eq:def-grad-est}). The convergence results versus number of function evaluations are shown in Figure \ref{fig:sample}. Some observations from Figure \ref{fig:sample} include
\begin{itemize}[align=left,leftmargin=*] 
    \item Given the same learning rate $\eta_t = \eta$ and the same number of function evaluations, in terms of \emph{function evaluations}, the sequence $ \{ f (\x_t) \}_t $ converges faster when $k$ is smaller. 
    \item Given the same learning rate $\eta_t = \eta$ and the same number of function evaluations, the sequence $ \{ \| \x_t - \x_\infty \| \}_t $ converges fastest when $ k = 10 $ or $k = 1$ for both $F_1$ and $F_2$. This is an intriguing observation, and suggests that there might be some fundamental relation between number of function evaluations needed and convergence of $ \{ \| \x_t - \x_\infty \| \}_t $. 
    \item When measured against number of function evaluations, $ \{ f (x_t) \}_t $ converges faster than $\{ \| \x_t - \x_\infty \| \}_t$. This is similar to the observations in Figure \ref{fig:F}. 
\end{itemize}


\begin{figure}[] 
    \centering 
    \subfloat[]{
    \includegraphics[width = 0.32\textwidth]{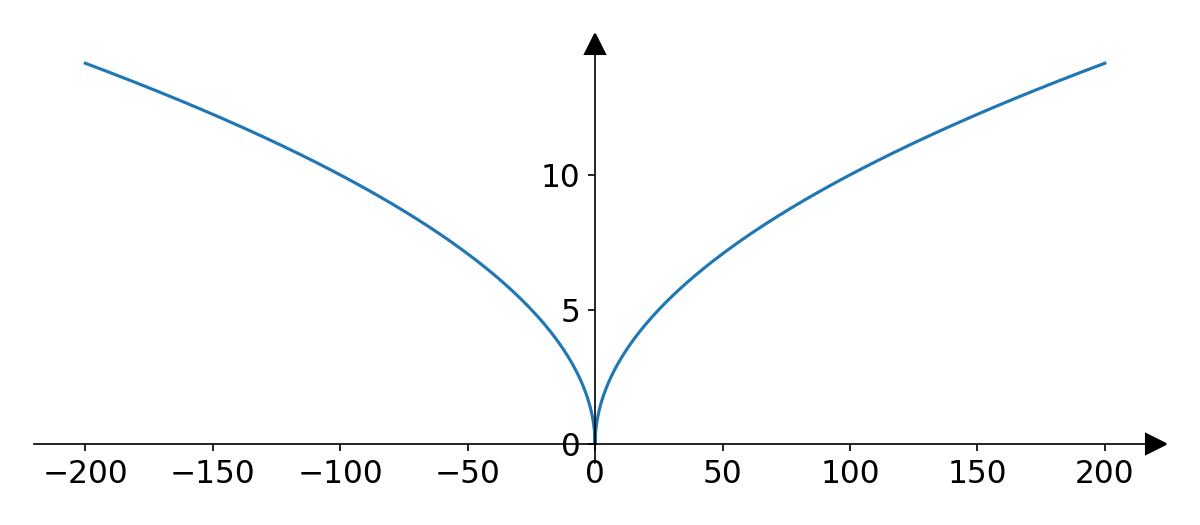} \label{fig:toy}}
    \subfloat[]{\includegraphics[width= 0.32\textwidth]{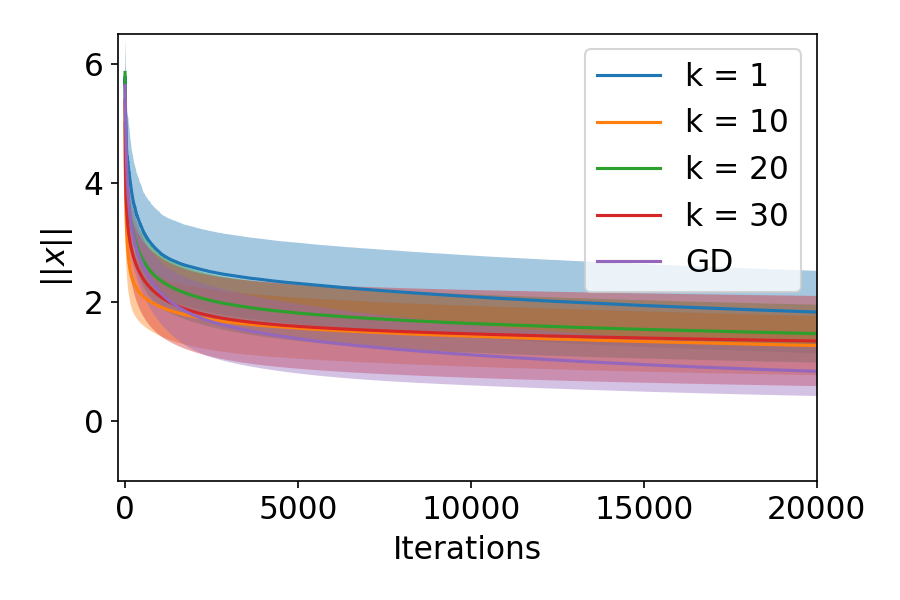} \label{fig:f1-1}}\,
    \subfloat[]{\includegraphics[width= 0.32\textwidth]{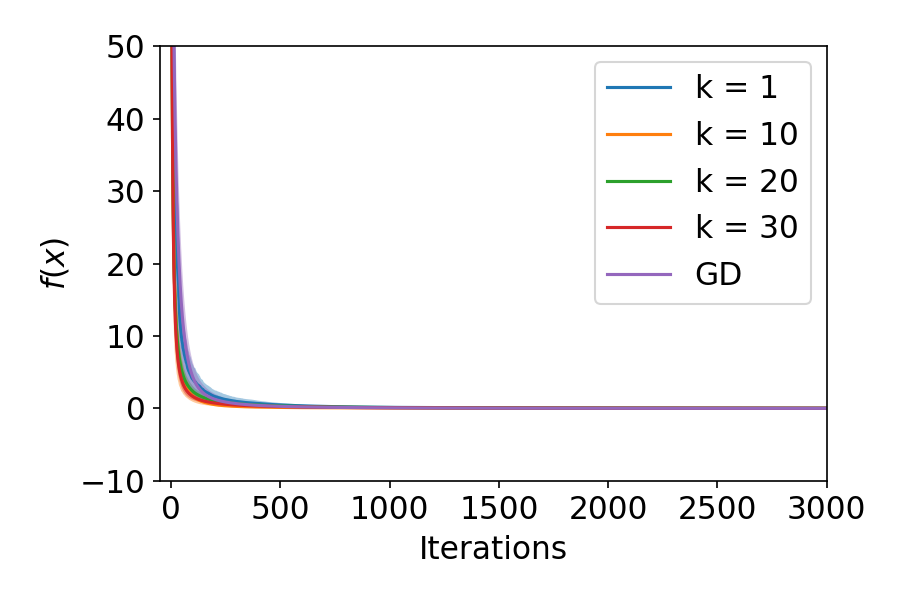}\label{fig:f1-2}} \\ 
    \subfloat[]{\includegraphics[width= 0.32\textwidth]{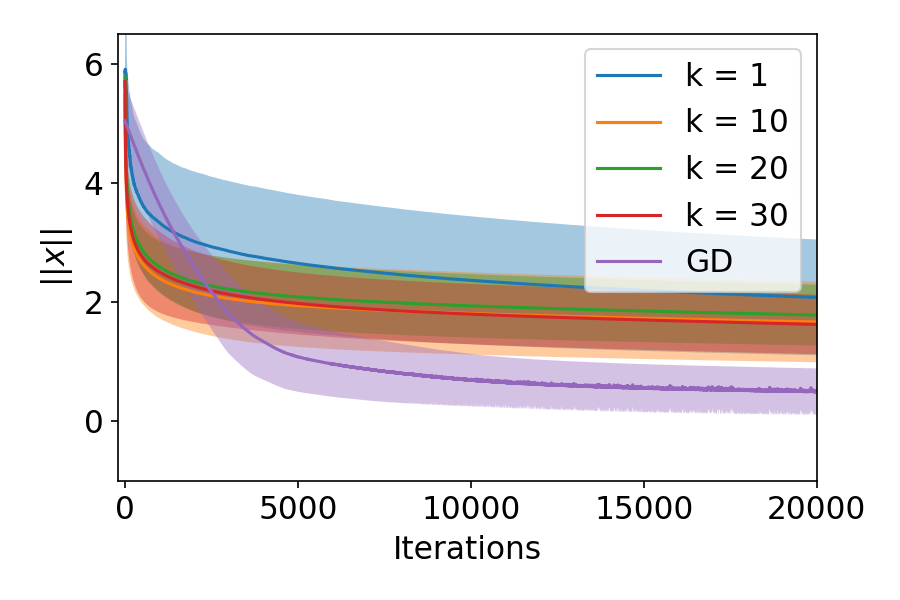}\label{fig:f2-1}} \,
    \subfloat[]{\includegraphics[width= 0.32\textwidth]{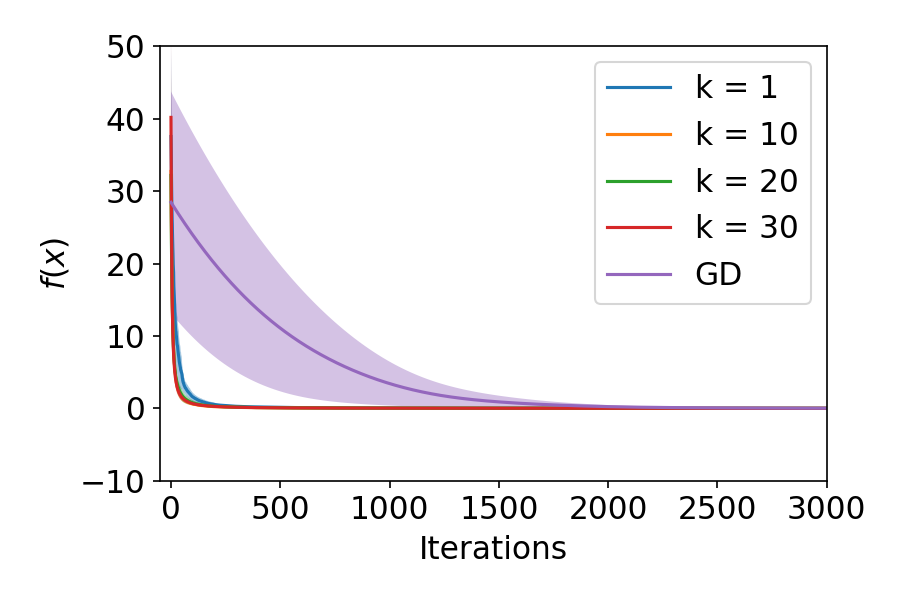}\label{fig:f2-2}} 
    \caption{Subfigure (a) plots the function $f (x) = \sqrt{|x|}$. Subfigures (b) and (c) (resp. (d) and (e)) plot results of SZGD and GD on test function $F_1$ (resp. $F_2$). The lines labeled with $k = 1$ (resp. $k=10$, etc.) show results of SZGD with $k = 1$ (resp. $k=10, etc.$). The line labeled GD plots results of the gradient descent algorithm. For all experiments, the step size $\eta_t$ is set to $0.005$ for all $t$. 
    Subfigures (b) and (d) show observed convergence rate results for $\{ \| \x_t - \x_\infty \|\}_t$ ($x_\infty = 0$); Subfigure (c) (resp. (e)) shows observed convergence rate results for $\{ \| F_1 (x_t) - F_1 (x_\infty) \|\}_t$ (resp. $\{ \| F_2 (x_t) - F_2 (x_\infty) \|\}_t$). 
    The solid lines show average results over 10 runs. The shaded areas below and above the solid lines indicate 1 standard deviation around the average. Since the starting point $\x_0$ is random, the trajectory of GD is also random.\label{fig:F}} 
\end{figure}

\begin{figure}[] 
    \centering 
    \subfloat[]{\includegraphics[width= 0.24\textwidth]{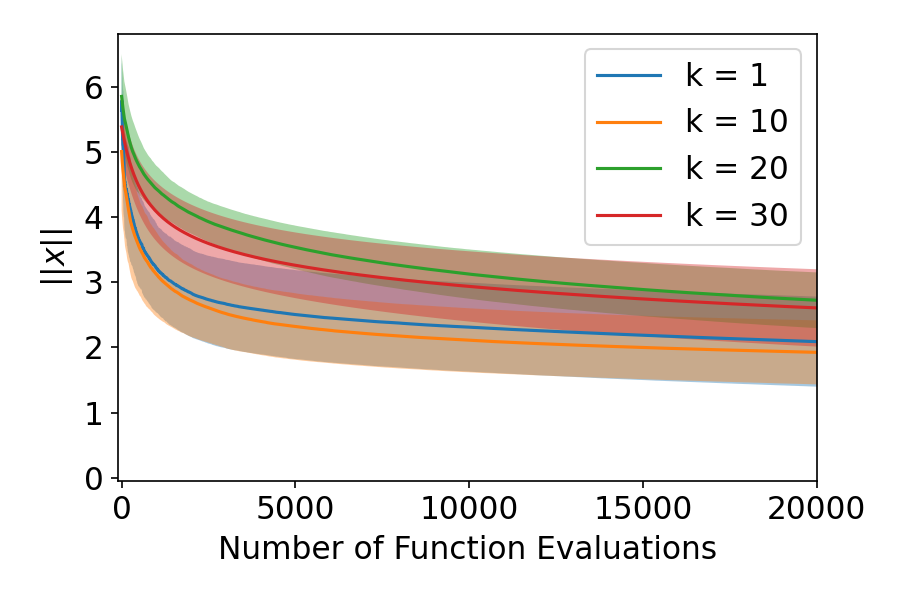}} 
    \subfloat[]{\includegraphics[width= 0.24\textwidth]{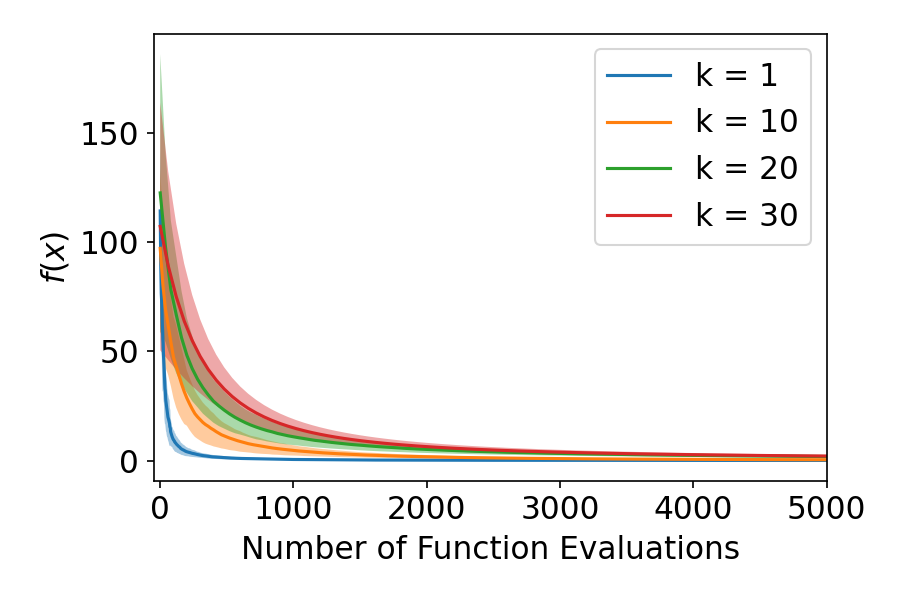}} 
    \subfloat[]{\includegraphics[width= 0.24\textwidth]{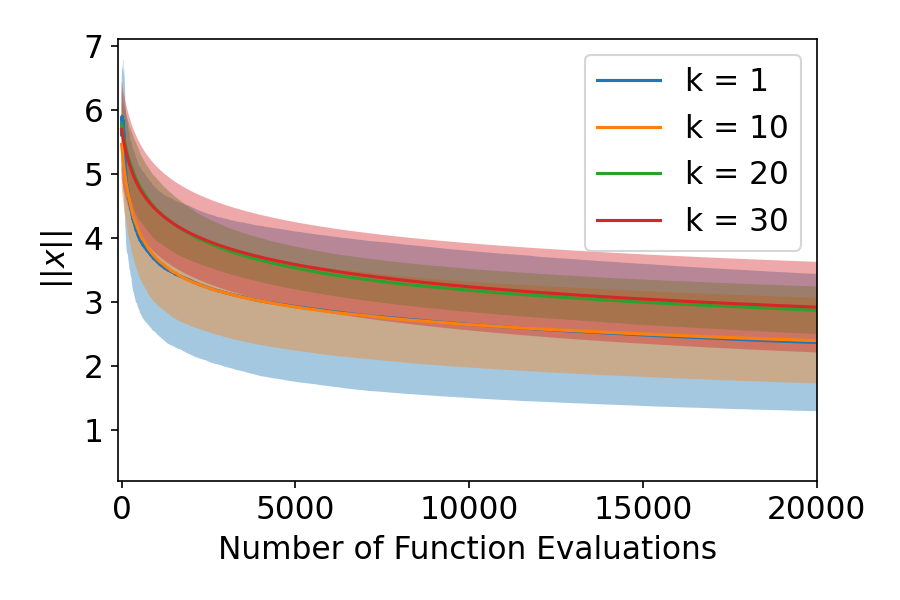}} 
    \subfloat[]{\includegraphics[width= 0.24\textwidth]{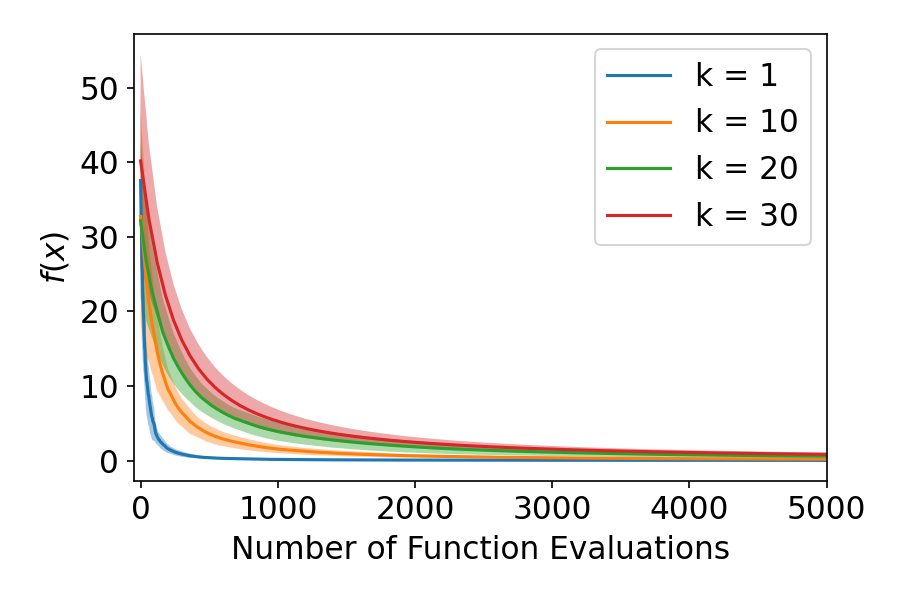}}
    \caption{Results of SZGD and GD on test functions $F_1$ (subfigures (a) and (b)) and $F_2$ (subfigures (c) and (d)). The solid lines show average results over 10 runs. The lines labeled with $k = 1$ (resp. $k=10$, etc.) show results of SZGD with $k = 1$ (resp. $k=10, etc.$). Subfigures (a) and (c) show observed convergence rate results for $\{ \| \x_t - \x_\infty \|\}_t$ ($\x_\infty = 0$); Subfigures (b) and (d) show observed convergence rate results for $\{ \| F_2 (\x_t) - F_2 (\x_\infty) \|\}_t$ ($F_2 (\x_\infty) = 0$). The step size $\eta$ is set to $0.005$. The shaded areas below and above the solid lines indicate 1 standard deviation around the average. Unlike Figure \ref{fig:F}, this figure plots errors (either $\| \x_t - \x_\infty \|$ or $ f( \x_t ) - f ( \x_\infty) $) against number of function evaluations. For example, when $k = 10$, one iteration requires $10 \times 2 = 20$ functions evaluations. \label{fig:sample}} 
\end{figure}

\section{Proof of Theorem \ref{thm:variance}}
\label{sec:proof}

To start with, we need the following facts in Propositions \ref{prop:uniform} and \ref{prop:sphere-exp}. 

\begin{proposition}[\cite{chikuse2003mani}]  
    \label{prop:uniform} 
    Let $ \V := [ \v_1, \v_2, \cdots, \v_k] \in \R^{n \times k}$ be uniformly sampled from the Stiefel manifold $\text{St}(n,k)$. Then the marginal distribution for any $\v_i$ is uniform over the unit sphere $\S^{n-1}$. 
\end{proposition} 

In words, the uniform measure over the Stiefel manifold $\text{St} (n,k)$ can be decomposed into a wedge product of the spherical measure over $ \S^{n-1} $ and the uniform measure over $ \text{St} (n-1,k-1) $ \cite{chikuse2003mani}. 
Another useful fact is the following proposition. 
\begin{proposition}
    \label{prop:sphere-exp}
    Let $\v$ be a vector uniformly randomly sampled from the unit sphere $\S^{n-1}$. Then it holds that
        $\E \[ \v \v^\top \] = \frac{1}{n} \mathbf{I} ,$ 
    where $\mathbf{I} \in \mathbb{R}^{n \times n}$ is the identity matrix. 
\end{proposition}


\begin{proof}
    Let $ v_i $ be the $i$-th entry of $\v$. 
    For any $a \in [-1,1]$ and $ i \neq j$, it holds that $ \E \[ v_i v_j \] = \E \[ v_i | v_j = a \] = 0 $. Thus $ \E \[ v_i v_j \] = 0 $ for $i \neq j$. Also, it holds that 
        $ 1 = \E \[ \| \v \|^2  \] = \sum_{i=1}^n \E \[ v_i^2 \], $
    which concludes the proof since $ \E \[ v_i^2 \] = \E \[ v_j^2 \] $ for any $i,j =1,2,\cdots,n$ (by symmetry). 
\end{proof} 



\begin{proof}[Proof of Theorem \ref{thm:variance}] 
    Since $ f (\x)$ is $L$-smooth ($\nabla f (\x)$ is $L$-Lipschitz), $\nabla^2 f (\x) $ (the weak total derivative of $ \nabla f (\x) $) is integrable. Let $\v \in \R^n$ be an arbitrary unit vector. When restricted to any line along direction $ \v\in \R^n $, it holds that $ \v^\top \nabla^2 f (\x) \v $ (the weak derivative of $ \v^\top \nabla (\x) $ along direction $\v$) has bounded $L_\infty$-norm. This is due to the fact that Lipschitz functions on any closed inteval $[a,b]$ forms the Sobolev space $ W^{1,\infty} [a,b] $. 
    
    Next we look at the variance bound for the estimator. Without loss of generality, we let $\x = \0$. Bounds for other values of $\x$ can be similarly obtained.

    Taylor's expansion of $ f $ with integral form gives 
    \begin{align*} 
        f (\delta \v_i) 
        =& \; 
        f (\0) + \delta \v_i^\top \nabla f (\0) + \int_{0}^{\delta} (\delta - t) \v_i^\top \nabla^2 f ( t \v_i ) \v_i \, dt  
    \end{align*} 
    
    Thus for any $\v_i \in \S^{n-1}$ and small $\delta$, 
    \begin{align*} 
        &\; \frac{1}{2} \big( f (\delta \v_i) - f (-\delta \v_i) \big) \\
        =& \;  
        \delta \v_i^\top \nabla f (\0) + \frac{1}{2} \int_{0}^{\delta} (\delta - t) \v_i^\top \nabla^2 f ( t \v_i ) \v_i \, dt - \frac{1}{2} \int_{0}^{-\delta} (-\delta - t) \v_i^\top \nabla^2 f ( t \v_i ) \v_i \, dt . 
    \end{align*} 
    For simplicity, let $ R_i = \int_{0}^{\delta} (\delta - t) \v_i^\top \nabla^2 f ( t \v_i ) \v_i \, dt - \int_{0}^{-\delta} (-\delta - t) \v_i^\top \nabla^2 f ( t \v_i ) \v_i \, dt  $, and Cauchy--Schwarz inequality gives 
    \begin{align*} 
        | R_i | 
        \le& \;  
        \frac{1}{2} \( \int_{0}^{\delta} (\delta - t)^2 \, dt  \)^{1/2} \( \int_0^\delta \( \v_i^\top \nabla^2 f ( t \v_i ) \v_i \)^2 \, dt \)^{1/2} \\
        & + \frac{1}{2} \( \int_{-\delta}^0 (-\delta - t)^2 \, dt \)^{1/2} \( \int_{-\delta}^0 \( \v_i^\top \nabla^2 f ( t \v_i ) \v_i \)^2 \, dt \)^{1/2} 
        \le 
        \frac{2L}{\sqrt{3}} \delta^2 
    \end{align*} 
    for all $i = 1,2,\cdots,k$. 
    For any $i,k,n$, it holds that 
    \begin{align*} 
        &\; \E \[ \left\| \frac{1}{2} \( f (\delta \v_i) - f (-\delta \v_i) \) \v_i \right\|^2 \] - \left\| \E \[ \frac{ \sqrt{k}}{2} \( f (\delta \v_i) - f (-\delta \v_i) \) \v_i \] \right\|^2 \\ 
        =& \; 
        \E \[ \left\| \( \delta \v_i^\top \nabla f (\0) + R_i \) \v_i \right\|^2 \] 
        - 
        k \left\| \E \[ \( \delta \v_i^\top \nabla f (\0) + R_i \) \v_i \]  \right\|^2 \\ 
        \overset{\textcircled{1}}{=}& \; 
        \E \[ \delta^2 \nabla f (\0)^\top  \v_i \v_i^\top \nabla f (\0) + 2 \delta \nabla f (\0)^\top \v_i R_i + R_i^2 \] - 
        k \left\| \E \[ \delta \v_i \v_i^\top \nabla f (\0) + R_i \v_i \] \right\|^2 . 
    \end{align*} 
    
    Since $ \E \[ \v_i \v_i^\top \] = \frac{1}{n} I$ (Propositions \ref{prop:uniform} and \ref{prop:sphere-exp}), \textcircled{1} gives 
    \begin{align*} 
        & \; \E \[ \left\| \frac{1}{2} \( f (\delta \v_i) - f (-\delta \v_i) \) \v_i \right\|^2 \] - \left\| \E \[ \frac{ \sqrt{k} }{2} \( f (\delta \v_i) - f ( - \delta \v_i) \) \v_i \] \right\|^2 \\ 
        =& \; 
        \frac{ \delta^2 }{ n } \| \nabla f (\0) \|^2 + 2 \delta \nabla f (\0)^\top \E \[ R_i \v_i \] + \E \[ R_i^2 \] 
        - 
        k \left\| \E \[ \frac{ \delta}{n} \nabla f (\0) + R_i \v_i \] \right\|^2 \\ 
        =& \; 
        \( \frac{\delta^2}{n}  - \frac{\delta^2 k }{n^2} \) \| \nabla f (\0) \|^2 + \( 2 \delta  - \frac{2 \delta k }{n} \) \nabla f (\0)^\top \E \[ R_i v_i \] 
        + \E \[ R_i^2 \] - k \left\| \E \[ R_i \v_i \] \right\|^2 \\ 
        \overset{\textcircled{2}}{\le}& \; 
        \( \frac{\delta^2}{n}  - \frac{ \delta^2 k }{n^2} \) \| \nabla f (\0) \|^2 + \frac{ 4 L \delta^3}{ \sqrt{3} } \( 1 - \frac{ k }{n} \) \| \nabla f (\0) \| 
        + 
        \frac{ 4 L^2 \delta^4 }{3} . 
    \end{align*}
    
    For the variance of the gradient estimator, we have 
    \begin{align*} 
        &\; \E \[ \left\| \wh{\nabla} f_k^\delta (\0) - \E \[ \wh{\nabla} f_k^\delta (\0) \] \right\|^2 \] \\
        =& \; 
        \E \[ \left\| \wh{\nabla} f_k^\delta (\0) \right\|^2 \] - \left\| \E \[ \wh{\nabla} f_k^\delta (\0) \] \right\|^2 \\ 
        =& \; 
        \E \[ \left\| \frac{n}{2 \delta k} \sum_{i=1}^k \( f (\delta \v_i) - f (\delta \v_i) \) \v_i \right\|^2 \] 
        - \left\| \frac{n}{2 \delta k} \sum_{i=1}^k \E \[ \( f (\delta \v_i) - f (-\delta \v_i) \) \v_i \] \right\|^2 \\ 
        \overset{\textcircled{3}}{=}& \; 
        \frac{n^2}{\delta^2 k^2} \sum_{i=1}^k \E \[ \left\| \frac{1}{2} \( f (\delta \v_i) - f (\delta \v_i) \) \v_i \right\|^2 \] \\
        &- \frac{ n^2 }{4 \delta^2 k^2 } \sum_{i,j=1}^k \E \[ \( f (\delta \v_i) - f (-\delta \v_i) \) \v_i \]^\top \E \[ \( f (\delta \v_j) - f (-\delta \v_j) \) \v_j \] , 
    \end{align*} 
    where the last equation follows from the orthonormality of $\{ \v_1, \v_2. \cdots, \v_k \}$. By Proposition \ref{prop:uniform}, we know that 
    $ \E \[ \( f (\delta \v_i) - f (-\delta \v_i) \) v_i \] = \E \[ \( f (\delta \v_j) - f (-\delta \v_j) \) \v_j \]$ for all $i,j = 1,2,\cdots,k$. Thus \textcircled{3} gives
    \begin{align*}
        & \; \E \[ \left\| \wh{\nabla} f_k^\delta (\0) - \E \[ \wh{\nabla} f_k^\delta (\0) \] \right\|^2 \] \\
        \overset{\textcircled{4}}{=}& \; 
        \frac{n^2}{\delta^2 k^2} \sum_{i=1}^k 
        \( \E \[ \left\| \frac{1}{2} \( f (\delta \v_i) - f (-\delta \v_i) \) v_i \right\|^2 \] - \left\| \E \[ \frac{ \sqrt{k} }{2} \( f (\delta \v_i) - f (-\delta \v_i) \) v_i \] \right\|^2 \).
    \end{align*}
    
    Combining \textcircled{2} and \textcircled{4} gives 
    \begin{align*}
        & \; \E \[ \left\| \wh{\nabla}f_k^\delta (\0) - \E \[ \wh{\nabla}f_k^\delta (\0) \] \right\|^2 \] \\
        =& \; 
        \frac{n^2}{\delta^2 k^2}\sum_{i=1}^k \(  \E \[ \left\| \frac{1}{ 2 } \( f (\delta \v_i) - f (\delta \v_i) \) v_i \right\|^2 \] - \left\| \E \[ \frac{ \sqrt{k} }{2 } \( f (\delta \v_i) - f (-\delta \v_i) \) v_i \] \right\|^2 \) \\ 
        \le& \; 
        \( \frac{n}{k} - 1 \) \| \nabla f (\0) \|^2 + \frac{ 4 L \delta}{ \sqrt{3} } \( \frac{n^2}{k}  - n \) \| \nabla f (\0) \| + \frac{ 4 L^2 n^2 \delta^2 }{ 3 k } . 
    \end{align*} 
\end{proof}




\section{Conclusion} 

This paper studies the SZGD algorithm and its performance on \Loj functions. In particular, we establish convergence rates for SZGD algorithms on \Loj functions. Our results show that access to noiseless zeroth-order oracle is sufficient for optimizing \Loj functions. We show that SZGD exhibits convergence behavior similar to its non-stochastic counterpart. Our results suggests $ \{ f (\x_t) \}_t $ tend to converge faster than $ \{ \| \x_t - \x_\infty \| \}_t $. We also observe some intriguing facts in the empirical studies. In particular, there might be an optimal choice of $k$ for SZGD to achieve a good convergence rate for $ \{ \| \x_t - \x_\infty \| \}_t $. 


\section*{Acknowledgement}

Tianyu Wang thanks Kai Du and Bin Gao for helpful discussions, and thanks Bin Gao for pointers to some important related works. 




\end{document}